\documentclass[11pt]{amsart}

\usepackage{hyperref}
\usepackage{amsfonts,amssymb,amscd,amsmath,amsthm,enumerate,mathrsfs}
\usepackage{graphicx,type1cm,eso-pic,enumerate}

\newcommand{\abs}{\vspace*{6pt}}

\DeclareMathOperator{\card}{card}

\DeclareMathOperator{\const}{const.}

\DeclareMathOperator{\Span}{span}

\DeclareMathOperator{\sign}{sign}

\DeclareMathOperator{\la}{\langle}
\DeclareMathOperator{\ra}{\rangle}

\newcommand{\skp}{{\langle .,. \rangle}}

\newcommand{\R}{\mathbb{R}}
\newcommand{\Z}{\mathbb{Z}}
\newcommand{\Q}{\mathbb{Q}}
\newcommand{\N}{\mathbb{N}}

\newcommand{\g}{\mathfrak{g}}

\renewcommand{\g}{\gamma}

\newcommand{\LL}{\mathcal{L}}

\newcommand{\TT}{\mathcal{T}}

\newcommand{\M}{\mathcal{M}}

\newcommand{\T}{{\mathbb{T}^2}}

\newcommand{\e}{\varepsilon}
\newcommand{\del}{\xi}

\newcommand{\al}{\alpha}
\newcommand{\om}{\omega}
\newcommand{\lam}{\lambda}
\newcommand{\sig}{\sigma}

\newcommand{\vf}{\varphi}

\theoremstyle{plain}
\newtheorem{defn}{Definition}[section]
\newtheorem{lemma}[defn]{Lemma}
\newtheorem{prob}[defn]{Problem}

\newtheorem{prop}[defn]{Proposition}
\newtheorem{thm}[defn]{Theorem}

\newtheorem{mainthm}[defn]{Main Theorem}

\theoremstyle{definition}
\newtheorem{remark}[defn]{Remark}

\newcommand{\mane}{Ma\~n\'e}
\newcommand{\Mane}{Ma\~n\'e }

\begin{document}

\hypersetup{pdftitle = {The stable norm on the 2-torus at irrational directions}, pdfauthor = {Stefan Klempnauer, Jan Philipp Schr\"oder}}

\title[The stable norm on $\T$ at irrational directions]{The stable norm on the 2-torus at irrational directions}
\author[S. Klempnauer, J. P. Schr\"oder]{Stefan Klempnauer, Jan Philipp Schr\"oder}
\address{Faculty of Mathematics \\ Ruhr University \\ 44780 Bochum \\ Germany}
\email{\url{stefan.klempnauer@rub.de}, \url{jan.schroeder-a57@rub.de}}
\date{\today}
\keywords{Finsler metric, stable norm, Mather's action functional, minimal geodesic, KAM-torus, hyperbolicity}

\begin{abstract}
 We study the structure of the stable norm of Finsler metrics on the 2-torus with a focus to points of irrational slope. By our results, the stable norm detects KAM-tori and hyperbolicity in the geodesic flow. Moreover, we study the stable norm in some natural examples.
\end{abstract}

\maketitle



\section{Introduction and main results}

In this paper we study properties of Finsler metrics $F:T\T\to\R$ on the 2-torus $\T=\R^2/\Z^2$, see \cite{BCS} for information on Finsler metrics. The Finsler metrics are not assumed to be reversible, such that our results apply to general Tonelli Lagrangians $L:T\T\to\R$, see \cite{cipp}. Readers unfamiliar with Finsler metrics may think of the norm $F(v)=\sqrt{g(v,v)}$ of a Riemannian metric $g$ in $\T$.

The object we study is the marked length spectrum $\sig_F$. We write
\[ \textstyle l_F(c;[a,b]) = \int_a^b F(\dot c) dt \]
for the $F$-length structure. Identifying $\Z^2\cong \pi_1(\T)$ the {\em marked length spectrum} is defined as
\[ \sig_F:\Z^2\to\R, \quad \sig_F(z) := \inf \left\{ l_F(c) : \text{the homotopy class of $c$ is } [c]=z \right\} . \]
Thus, $\sig_F$ contains information on closed $F$-geodesics.

We extend $\sig_F$ to a norm on $\R^2$. By the results of Hedlund \cite{hedlund} (see Theorem \ref{morse periodic} \eqref{morse periodic item 1} below), $\sig_F$ is positively homogeneous:
\[ \sig_F(a \cdot z) = a \cdot \sig_F(z) \qquad \forall z\in \Z^2,~ a \in \N_0 . \]
Moreover, using the fact that lifts of closed curves in linearly independent homotopy classes intersect in the universal cover $\R^2$, one infers
\[ z,w\in \Z^2 \text{ linearly independent} \quad \implies\quad \sig_F(z +w) < \sig_F(z)+\sig_F(w) .\]
Extending $\sig_F$ first homogeneously along lines of rational slope and then continuously to $\R^2$, we obtain a convex (in general non-reversible) norm
\[ \sig_F:\R^2\to\R \]
called the {\em stable norm} (see Section 2 of \cite{bangert-minimal-geod}). Note that $\sig_F$ is related to {\em Mather's $\beta$-function} 
\[ \beta_F:= \frac{1}{2}\sig_F^2 \]
of the ``Tonelli'' Lagrangian $L=\frac{1}{2}F^2$ (see Section 1 of \cite{massart-thesis}).

Let us recall the classical result on $\sig_F$ due to J. Mather \cite{mather1}, which is the starting point for our work. The theorem has a different proof due to V. Bangert \cite{bangert1}. (These authors prove the theorem for reversible Finsler metrics, while for non-reversible Finsler metrics the results are also true, see \cite{paper1}.) 

\begin{thm}[Mather, Bangert]\label{thm intro mather-bangert}
 Let $F$ be any Finsler metric on $\T$ with stable norm $\sig_F$ and $\xi=(\xi_1,\xi_2)\in \R^2-\{0\}$.
 \begin{enumerate}[(i)]
  \item\label{thm intro mather-bangert irrat} If $\xi$ has irrational slope $\xi_2/\xi_1\in\R-\Q$, then the stable norm $\sig_F$ is differentiable at $\xi$.
  
  \item\label{thm intro mather-bangert rat} If $\xi$ has rational or infinite slope $\xi_2/\xi_1\in\Q\cup\{\infty\}$, then the stable norm $\sig_F$ is differentiable at $\xi$ if and only if there exists a foliation of $\T$ by shortest closed geodesics in the homotopy class $z$, where $z$ is the prime element in $\Z^2\cap \R_{>0}\xi$.
 \end{enumerate}
\end{thm}

\begin{remark}
 There is a well-known rigidity phenomenon. Let us write
 \[ S\T=\{F=1\}\subset T\T \]
 for the unit tangent bundle and
 \[ \phi_F^t:S\T\to S\T \]
 for the geodesic flow of $F$. Then $\sig_F:\R^2-\{0\}\to\R$ is $C^1$ if and only if the geodesic flow $\phi_F^t$ of $F$ is $C^0$-integrable in $S\T$, that is, $S\T$ is $C^0$-foliated by invariant graphs \cite{massart-sorrentino}. In the Riemannian case, the $C^0$-integrability of $\phi_F^t$ is equivalent to the flatness of the metric $F$ by a classical result of E. Hopf \cite{hopf}.
 
 In general, two Finsler or Riemannian metrics with the same stable norm need not be isometric. See also the discussion in Section 6 of \cite{bangert1}.
\end{remark}

Theorem \ref{thm intro mather-bangert} \eqref{thm intro mather-bangert rat} suggests a deeper relationship between the structure of the geodesic flow $\phi_F^t:S\T\to S\T$ and the stable norm $\sig_F$. As $\sig_F$ is defined in terms of minimizers of a variational functional, we will look for relations to the minimal geodesics of $F$. Recall that a {\em minimal geodesic} is a geodesic $c:\R\to\T$ with the property that the lifts $\tilde c:\R\to\R^2$ to the universal cover minimize the length between any two of their points. Writing $d_F$ for the (in general non-symmetric) distance induced by the length $l_F$ on the universal cover $\R^2$, this means for the lifts $\tilde c$, that
\[ l_F(\tilde c;[a,b]) = d_F(\tilde c(a),\tilde c(b)) \qquad \forall a\leq b . \]
For a point $\xi\in S^1$ we define
\[ \M(\xi)\subset S\T \]
to be the set of initial conditions of minimal geodesics $c:\R\to\T$ with asymptotic direction $\delta^+(c)=\xi$, where
\[ \delta^+(c) := \lim_{t\to\infty}\frac{\tilde c(t)}{|\tilde c(t)|} , \]
$\tilde c:\R\to\R^2$ being any lift of $c$ and $|.|$ the euclidean norm on $\R^2$. It is known that the above limit exists for all minimal geodesics. We shall write $\M(a\cdot \xi)=\M(\xi)$ for $a>0$. It is furthermore known that the shortest closed geodesics in the definition of $\sig_F(z)$ lie in the set $\M(z)$ for $z\in\Z^2$. For these and more facts we refer to \cite{hedlund} and \cite{bangert} in the Riemannian case and to \cite{zaustinsky} and \cite{paper1} for the general Finsler case.

Motivated by Theorem \ref{thm intro mather-bangert}, we state the following problem.

\begin{prob}
 Relate the properties of $\sig_F$ at a given point $\xi\in \R^2-\{0\}$ to the structure of the set $\M(\xi)\subset S\T$.
\end{prob}

A particularly nice structure of $\M(\xi)$ would be that it is a KAM-torus.

\begin{defn}\label{def KAM intro}
 Let $\phi^t:X\to X$ be a $C^\infty$-flow on a $C^\infty$-manifold $X$. A {\em $C^k$-KAM-torus} (of dimension $n$) is a $C^k$-submanifold $\TT\subset X$, such that
 \begin{enumerate}[(i)]
  \item $\TT$ is invariant under $\phi^t$,
  
  \item there exists a $C^k$-diffeomorphism $\TT\to \mathbb{T}^n = \R^n/\Z^n$ conjugating $\phi^t|_\TT$ to a linear flow $\psi^t$ on $\mathbb{T}^n$ of the form
  \[ \psi^t(x) = x+t\rho \mod \Z^n . \]
 \end{enumerate}
\end{defn}

In our case $X=S\T$ and $\phi^t=\phi_F^t$ we consider only KAM-tori of dimension $n=2$ (half the dimension of the symplectic manifold $T^*\T$). 

Another possible structure of $\M(\xi)$ would be hyperbolicity.

\begin{defn}\label{def hyp intro}
 Let $\phi^t:X\to X$ be a $C^\infty$-flow on a $C^\infty$-manifold $X$. A subset $\Lambda\subset X$ is {\em uniformly hyperbolic} for $\phi^t$, if there exist constants $C,\lam>0$ and distributions $\{E^s(x)\}_{x\in\Lambda},\{E^u(x)\}_{x\in\Lambda}$, such that
 \begin{enumerate}[(i)]
  \item $\Lambda$ is compact and $\phi^t$-invariant,
  
  \item the distributions are $\phi^t$-invariant:
  \[ D\phi^t(x)E^s(x) = E^s(\phi^t x) , \qquad  D\phi^t(x)E^u(x) = E^u(\phi^t x) , \]
  
  \item the distributions together with the flow direction span the tangent \linebreak spaces:
  \[ T_xX = \R\cdot (\tfrac{d}{dt}\big|_{t=0}\phi^tx) \oplus E^s(x) \oplus E^u(x) , \]
  
  \item with respect to some Riemannian metric on $X$ we have contraction:
  \begin{align*}
 & \| D\phi^t(x) v \| \leq C \cdot \exp(-\lam t) \cdot \|v\| \qquad \forall  t\geq 0, v\in E^s(x) , \\
 & \| D\phi^{-t}(x) v \| \leq C \cdot \exp(-\lam t) \cdot \|v\| \qquad \forall  t\geq 0, v\in E^u(x) .
\end{align*}
 \end{enumerate}
\end{defn}

\pagebreak

We can now state our main result concerning the structure of the stable norm at points of irrational slope.

\begin{mainthm}\label{main thm intro irrat}
 Let $F$ be any Finsler metric on $\T$ with stable norm $\sig_F$ and let $\xi \in \R^2-\{0\}$ have irrational slope $\xi_2/\xi_1\in \R-\Q$.
 \begin{enumerate}[(i)]
  \item\label{main irrat i} If the set $\M(\xi)\subset S\T$ is a $C^3$-KAM-torus for the geodesic flow $\phi_F^t$, then the square of the stable norm $\sig_F$ is strongly convex near $\xi$. More precisely, there exists a constant $C>0$, such that Mather's $\beta$-function
 \[ \beta_F = \frac{1}{2} \sig_F^2 \]
 satisfies for all $v\in\R^2$ the estimate
 \begin{align*}
   \beta_F(\xi+v) - \beta_F( \xi) - D\beta_F( \xi)[ v]\geq C \cdot| v|^2 .
 \end{align*}
 
 \item\label{main irrat ii} Suppose that in each non-trivial free homotopy class of $\T$, there exists only one shortest closed $F$-geodesic. If the set $\M(\xi)\subset S\T$ is uniformly hyperbolic for the geodesic flow $\phi_F^t$, then the stable norm $\sig_F$ is exponentially flat near $\xi$. More precisely, there exist constants $C,\lam>0$, such that in all choices of rays $R \subset \R^2$ emanating from the origin there exist sequences $v_n\to 0, v_n\neq 0$, so that
 \begin{align*}
  & \sig_F(\xi+v_n) - \sig_F(\xi) - D\sig_F(\xi)[v_n] \leq |v_n|^{1/4} \cdot C \cdot \exp\left(-\lam \cdot \frac{1}{| v_n |^{1/4}} \right) .
 \end{align*}
 \end{enumerate}
\end{mainthm}

\begin{remark}
 Note that the differentiability of $\sig_F$ and $\beta_F$ at $\xi$ follows directly from Theorem \ref{thm intro mather-bangert} \eqref{thm intro mather-bangert irrat}. In Main Theorem \ref{main thm intro irrat} \eqref{main irrat i}, we will also show that for some $C'>0$ and all $v$
 \[ \beta_F(\xi+v) - \beta_F( \xi) - D\beta_F( \xi)[ v] \leq C' \cdot |v|^2 . \]
 As to part \eqref{main irrat ii}, note that by convexity we have
 \begin{align*}
  & 0 \leq \sig_F(\xi+v) - \sig_F(\xi) -D\sig_F(\xi)[v] .
 \end{align*}
 The function
 \[ t> 0 \quad \mapsto \quad t^{1/4} \cdot C \cdot \exp\left(-\lam \cdot \frac{1}{t^{1/4}} \right) \]
 vanishes in $t=0$ to infinite order.
\end{remark}

If one draws the unit circle of $\sig_F$ with a computer, it looks like a straight line near $\xi$, if $\M(\xi)$ is hyperbolic. If, on the other hand $\M(\xi)$ is a KAM-torus, it will look like a parabola as in the euclidean (integrable) case. We will give some examples of Finsler metrics together with their stable norms in Section \ref{section examples} below. Intuitively, one sees that when perturbing the euclidean metric with the stable norm being again the euclidean metric, the convexity of the stable norm moves notably into vertices at rational directions, while at directions irrational slope the unit circle looks more and more like a straight line. Indeed, in Figure \ref{fig_rivin} below the unit circle of the stable norm in the case of a hyperbolic metric on the (punctured) torus looks polygonal, even though it is strictly convex.

\pagebreak

Theorem \ref{thm intro mather-bangert} \eqref{thm intro mather-bangert rat} can be rephrased as
\begin{itemize}
 \item $\sig_F$ is differentiable at a point $\xi\in \R^2-\{0\}$ of rational slope if and only if the set $\M(\xi)$ is a $C^0$-KAM-torus.
\end{itemize}
The following problem arises.

\begin{prob}\label{quest irrat}
 Give a criterion on the stable norm $\sig_F$ near a given irrational direction $\xi \in \R^2-\{0\}$, which is equivalent to the case where the set $\M(\xi)$ is a KAM-torus for the geodesic flow $\phi_F^t$.
\end{prob}

We saw in Theorem \ref{thm intro mather-bangert} that the answer in the rational case was given in terms of the differentiability of $\sig_F$, while for the irrational case this is not possible since $\sig_F$ is always differentiable here. In this light, Main Theorem \ref{main thm intro irrat} partially answers Problem \ref{quest irrat} in terms of the flatness properties of $\sig_F$ near $\xi$. However, there are several issues to discuss.

Let us start with the condition in Main Theorem \ref{main thm intro irrat} \eqref{main irrat ii}, that each free homotopy class contains only one shortest closed geodesic. This is certainly not fulfilled for every Finsler metric.

\begin{defn}\label{def generic}
 A property of Finsler metrics is said to be {\em conformally generic} if, given an arbitrary Finsler metric $F_0$ on $\T$, the property holds for all Finsler metrics $F$ of the form
 \[ F(x,v) = f(x) \cdot F_0(x,v) \]
 with $f$ belonging to a residual subset of
 \[ \{ f: \T\to \R : \text{$f>0$ everywhere and $f$ is $C^\infty$} \} \]
 in the $C^\infty$-topology. Here, a residual set in a topological space is a countable intersection of open and dense subsets.
\end{defn}

In the Lagrangian setting, the above notion of genericity is related via Maupertuis' principle to R. \mane's way of perturbing a Tonelli Lagrangian $L_0$ by a potential into $L(x,v)=L_0(x,v)+f(x)$, see \cite{mane}.

The next proposition is proved in \cite{generic-paper}.

\begin{prop}\label{generic-paper result}
 The property to admit only one shortest closed geodesic in each free homotopy class is conformally generic.
\end{prop}

This shows that Main Theorem \ref{main thm intro irrat} applies to ``most'' Finsler metrics on $\T$ without the extra condition in item \eqref{main irrat ii}.

Next, let us see, what alternatives there are for the structure of $\M(\xi)$:
\begin{enumerate}[(A)]
 \item $\M(\xi)$ is a $C^3$-KAM-torus for $\phi_F^t$,
  
 \item $\M(\xi)$ is uniformly hyperbolic for $\phi_F^t$,
 
 \item none of the above two.
\end{enumerate}

First we note that case (A) occurs frequently by KAM-theory, if the Finsler metric $F$ is close to one with an integrable geodesic flow (e.g.\ the euclidean metric), see e.g.\ \cite{moser}, while we do not attempt to give a full overview on the literature on KAM-theory. The reason for us to use $C^3$-regularity is given in Remark \ref{remark kam torus c3}.

We will prove the following proposition concerning case (B).

\begin{prop}\label{prop lecalvez intro}
 The following property of Finsler metrics on $\T$ is conformally generic:
 \begin{itemize}
  \item For an open and dense subset $U\subset S^1$, the sets $\M(\xi)$ are uniformly hyperbolic for all $\xi\in U$. The set $U$ strictly contains all $\xi\in S^1$ with rational or infinite slope.
 \end{itemize}
\end{prop}

Put together, one expects that cases (A) and (B) occur quite frequently. Aiming at Problem \ref{quest irrat}, for conformally generic Finsler metrics, Main Theorem \ref{main thm intro irrat} tells us when we are in one of the cases (A) or (B), while it is not able to distinguish case (C).

Let us have a brief look at case (C). This case could be quite subtle and will be left for future research. See also the discussion in Section 10 of \cite{mckay}. This case contains the following situations:
\begin{itemize}
 \item[(CA)] $\M(\xi)$ is a $C^0$-KAM-torus, but not a $C^3$-KAM-torus,
 
 \item[(CB)] $\M(\xi)$ is not a $C^0$-KAM-torus, but also not uniformly hyperbolic.
\end{itemize}
As case (C) is excluded for rational $\xi$, if the Finsler metric is chosen generically (Proposition \ref{prop lecalvez intro}), let us assume that $\xi$ has irrational slope. One might expect that case (CA) can be treated as a generalization of case (A) with some degeneracy to be expected. In case (CB) the set $\pi(\M^{rec}(\xi))$ of recurrent minimal geodesics projected to $\T$ is nowhere dense in $\T$ \cite{bangert}. This case occurs for generic Finsler metrics, fixing $\xi\in S^1$ with slope $\xi_2/\xi_1$ a Liouville number, see \cite{mather2} (for Diophantine numbers, KAM-tori can occur by KAM-theory). Here, it is known that homoclinic behavior of geodesics close to $\M^{rec}(\xi)$ occurs, i.e.\ one could expect some hyperbolicity. However, it is still possible that $\pi(\M(\xi))=\T$; also, one can have vanishing or non-vanishing Lyapunov exponents (non-uniform hyperbolicity). All these topics will not be treated here; let us in this connection only refer to the work of M. C. Arnaud: \cite{arnaud}, \cite{arnaud-example}, \cite{arnaud non-hyp}.

\abs

Finally, we will prove a theorem on the structure of $\sig_F$ in rational directions. Recall that generically the set $\M(\xi)$ is hyperbolic for all $\xi$ with rational or infinite slope, see Proposition \ref{prop lecalvez intro}. In this case, we can sharpen Theorem \ref{thm intro mather-bangert} \eqref{thm intro mather-bangert rat}, obtaining an estimate analogous to Main Theorem \ref{main thm intro irrat} \eqref{main irrat ii}. We write
\[ D^+\sig_F(\xi)[v] :=  \inf_{t>0} \frac{\sig_F(\xi+tv)-\sig_F(\xi)}{t} = \lim_{t\searrow 0} \frac{\sig_F(\xi+tv)-\sig_F(\xi)}{t} \]
for the forward directional derivative of $\sig_F$. The second equality holds due to convexity.

\begin{thm}\label{intro rational hyp}
 Let $\xi\in \R^2 - \{0\}$ with rational or infinite slope. Assume that the set of periodic minimal geodesics $\M^{per}(\xi)\subset S\T$ is uniformly hyperbolic for the geodesic flow $\phi_F^t$. Then there exist constants $C,\lam, \e>0$, such that for all $v\in \R^2$ with euclidean norm $|v|\leq \e$
 \begin{align*}
  & \sig_F(\xi+v) - \sig_F(\xi) - D^+\sig_F(\xi)[v] \leq |v| \cdot C \cdot \exp\left( -\lam \cdot \frac{1}{|v|} \right) .
 \end{align*}
\end{thm}

\begin{remark}
 Note that by convexity we have for all $v\in\R^2$
 \begin{align*}
  & 0 \leq \sig_F(\xi+v) - \sig_F(\xi) -D^+\sig_F(\xi)[v] .
 \end{align*}
 Also note that the function
 \[ t> 0 \quad \mapsto \quad t\cdot C \cdot \exp\left(-\lam \cdot \frac{1}{t} \right) \]
 vanishes in $t=0$ to infinite order.
\end{remark}

Note that, intuitively, there are relations of Theorem \ref{intro rational hyp} and Main Theorem \ref{main thm intro irrat} \eqref{main irrat ii} to \cite{bressaus-quas} in the setting of ergodic optimization.

We close the introduction with a remark on the stable norm on higher genus surfaces.

\begin{remark}
 We saw that in the torus case, the stable norm contains much information on the dynamics of $\phi_F^t$. The natural question is, whether this is true also for higher genus surfaces. Here, there are results analogous to Theorem \ref{thm intro mather-bangert} due to D. Massart \cite{massart2} (note, however, the erratum \cite{massart-erratum}). In \cite{min_rays} the second author proves that a similar asymptotic object associated to $F$, namely the horofunction boundary is generically homeomorphic to that of a constant curvature metric. The following question should be an interesting topic for future research.
 
 {\bf Question.} Are the differentiability properties of the stable norm of a generic Finsler metric $F$ on a closed orientable surface $M$ of genus at least two the same as those of the stable norm of a constant curvature Riemannian metric?

\end{remark}

{\bf Structure of this paper.}
Main Theorem \ref{main thm intro irrat} \eqref{main irrat i} is proved in Section \ref{section kam-case}. The arguments for Main Theorem \ref{main thm intro irrat} \eqref{main irrat ii} and Theorem \ref{intro rational hyp} are contained in Section \ref{section hyp}. In Section \ref{section lecalvez}, we sketch the proof of Proposition \ref{prop lecalvez intro}; in Section \ref{section examples} we study some natural examples of Finsler metrics and their stable norms.

\section{The case of a KAM-torus}\label{section kam-case}

We fix the Finsler metric $F$. The associated sets $\M(\xi)$ can be seen as remnants of KAM-tori (recall Definition \ref{def KAM intro} for the definition of a KAM-torus). More precisely, if a KAM-torus $\TT\subset S\T$ is a Lipschitz graph over the base $\T$, then it is well-known (cf.\ Theorem 17.4 in \cite{mather-forni} or Section 3 in \cite{diss}) that $\TT\subset \M(\xi)$ for some $\xi$. In this section we fix $\xi\in \R^2-\{0\}$ and in order to prove Main Theorem \ref{main thm intro irrat} \eqref{main irrat i} we assume that
\begin{itemize}
 \item the set $\TT=\M(\xi)$ is a $C^k$-KAM-torus for the geodesic flow $\phi_F^t$, while $\phi_F^t|_\TT$ is conjugated via some diffeomorphism $\Phi:\TT\to\T$ to the linear flow $\psi^tx=x+t\rho$.
\end{itemize}
We shall call $\rho$ the {\em frequency vector} of $\M(\xi)$.

The aim of this section is to study the stable norm $\sig_F$ of $F$ close to $\xi$. We shall follow the ideas of K. F. Siburg, cf.\ \cite{siburg-paper} or Chapter 4 in \cite{siburg}. Note, however, that in our setting we do not need symplectic coordinate changes and the condition for our KAM-torus to be positive definite is fulfilled automatically, see below.

\begin{lemma}\label{lemma KAM-foliation}
 If the set $\M(\xi)$ is a $C^k$-KAM-torus for the geodesic flow $\phi_F^t$ as assumed above, then there exists a $C^k$-diffeomorphism $\vf:\T\to\T$, such that the push-forward $C^{k-1}$-Finsler metric $\vf_*F$ admits the straight lines
 \[ x + t \rho \mod \Z^2 \]
 as arc-length geodesics.
\end{lemma}

\begin{remark}\label{remark kam torus c3}
 For the push-forward $\vf_*F$ to be a Finsler metric, it should be at least $C^2$ away from the zero section, hence the KAM-torus in Lemma \ref{lemma KAM-foliation} should be at least $C^3$.
\end{remark}

\begin{proof}
 First we observe that the canonical projection $\pi|_{\M(\xi)}:\M(\xi) \to \T$ is a bi-Lipschitz homeomorphism. Indeed, by assumption all orbits in $\TT=\M(\xi)$ are recurrent under $\phi_F^t$, while it is known that $\pi$ restricted to the set of recurrent minimal geodesics $\M^{rec}(\xi)\subset \M(\xi)$ is a bi-Lipschitz homeomorphism onto its image in $\T$; for this let us refer to \cite{paper1}, in particular in the non-reversible Finsler case, while the extensive literature on the subject starts already with \cite{hedlund}. The claim follows.
 
 As $\pi:T\T\to\T$ is smooth and the $C^k$-submanifold $\M(\xi)$ is a Lipschitz-graph (the image of $\pi|_{\M(\xi)}^{-1}:\T\to T\T$), we find that $\pi|_{\M(\xi)}:\M(\xi)\to\T$ is a $C^k$-diffeomorphism. Consider the $C^k$-diffeomorphism $\Phi:\M(\xi)\to\T$ conjugating $\phi_F^t|_{\M(\xi)}$ to the linear flow $\psi^tx=x+t\rho$. Then for the geodesic $c_v:\R\to\T$ corresponding to $v\in \M(\xi)$ we find
 \begin{align*}
  c_v(t) & = \pi|_{\M(\xi)} \circ \phi_F^t(v) = \pi|_{\M(\xi)} \circ \Phi^{-1}\circ \psi^t \circ \Phi(v) \\
  & = \pi|_{\M(\xi)} \circ \Phi^{-1}( \Phi(v) + t \rho) .
 \end{align*}
 This shows that the $C^k$-diffeomorphism
 \[ \vf := \Phi \circ \pi|_{\M(\xi)}^{-1} : \T\to\T , \]
 sends the $F$-geodesics from $\M(\xi)$ to the desired straight lines.
\end{proof}

We write
\[ \hat F := \vf_*F :T\T\to\R \]
for the push-forward Finsler metric found in Lemma \ref{lemma KAM-foliation}. Let us see how the stable norm $\sig_F$ transforms under $\vf$.

\begin{lemma}\label{lemma vf on sigma}
 If $\vf:\T\to\T$ is the diffeomorphism from Lemma \ref{lemma KAM-foliation}, then there exists a linear isomorphism $L_\vf:\R^2\to\R^2$ with
 \[ \sig_{\hat F} = \sig_F \circ L_\vf^{-1} . \]
 Moreover, there exists $\lam>0$ with
 \[ \rho = \lam \cdot L_\vf \xi . \]
\end{lemma}

\begin{proof}
 We defined $\sig_F$ as a function on $\R^2$, while it could be equivalently defined on the first homology group $H_1(\T,\R)\cong\R^2$ via
 \[ \sig_F(h) = \inf\left\{ \sum_{i=1}^k r_i l_F(c_i) ~\bigg|~ r_i\in\R, c_i \text{ closed curve in $\T$}, h = \sum_{i=1}^k r_i [c_i] \right\} , \]
 see Section 2 of \cite{bangert-minimal-geod}. The diffeomorphism $\vf:\T\to\T$ induces a linear isomorphism
 \[ L_\vf:H_1(\T,\R)\to H_1(\T,\R) , \qquad L_\vf[c] = [\vf \circ c] , \]
 see Corollary 4.3 on p.\ 176 in \cite{bredon}. Using that $l_{\hat F}(c)= l_F(\vf^{-1}\circ c)$ for the length of curves, we find with the above definition of $\sig_F$, that
 \begin{align*}
  \sig_{\hat F}(L_\vf h) & = \inf\left\{ \sum r_i l_{\hat F}(c_i) ~\bigg|~   L_\vf h = \sum r_i [\vf \circ\vf^{-1}\circ c_i] \right\} \\
  & = \inf\left\{ \sum r_i l_F(\vf^{-1}\circ c_i) ~\bigg|~   L_\vf h = \sum r_i L_\vf[\vf^{-1}\circ c_i] \right\} \\
  & = \inf\left\{ \sum r_i l_F(\vf^{-1}\circ c_i) ~\bigg|~   h = \sum r_i [\vf^{-1}\circ c_i] \right\} \\
  & = \sig_F(h), 
 \end{align*}
 i.e.\ the first claim follows. Let now $v\in \M(\xi)$, then the $F$-geodesic $c_v$ is recurrent and there exists a sequence $T_n\to\infty$ with $c_v(T_n) \to c_v(0)$. We close $c_v|_{[0,T_n]}$ by a short segment $\e_n$ and write $\widetilde{c_v}:\R\to\R^2$ for some lift of $c_v$. It follows for the homology class of $c_v|_{[0,T_n]} * \e_n$ seen as a point in $\Z^2\subset\R^2$, that
 \begin{align}\label{eqn hom class of c_v}
  \lim_{n\to\infty} [ c_v|_{[0,T_n]} * \e_n ] - \widetilde{c_v}(T_n)-\widetilde{c_v}(0) = 0.
 \end{align}
 By Birkhoff's ergodic theorem, the rotation vector
 \[ \rho_0 := \lim_{T\to\infty} \frac{\widetilde{c_v}(T)}{T} \]
 exists (at least for almost every $v\in \M(\xi)$), such that by definition of $\xi=\lim_{T\to\infty}\frac{\widetilde{c_v}(T)}{|\widetilde{c_v}(T)|}$, we find
 \begin{align*}
  \rho_0 & = \lim_{T\to\infty} \frac{\widetilde{c_v}(T)}{T} = \lim_{T\to\infty}  \frac{|\widetilde{c_v}(T)|}{T} \cdot \frac{\widetilde{c_v}(T)}{|\widetilde{c_v}(T)|} = : \lam \cdot \xi .
 \end{align*}
 Moreover, by \eqref{eqn hom class of c_v} and the analogous statement for $\vf\circ c_v(t) = x+t\rho$
 \begin{align*}
  L_\vf\rho_0 & = L_\vf\left(\lim_{T\to\infty} \frac{\widetilde{c_v}(T)}{T}\right) = L_\vf\left(\lim_{n\to\infty} \frac{[ c_v|_{[0,T_n]} * \e_n ]}{T_n}\right) \\
  & = \lim_{n\to\infty} \frac{[\vf \circ c_v|_{[0,T_n]} *\vf \circ  \e_n ]}{T_n} = \lim_{n\to\infty} \frac{\rho\cdot T_n }{T_n} = \rho .
 \end{align*}
 This proves the second claim.
\end{proof}

In the following, we shall use the standard coordinates $\T\times \R^2$ for $T\T$. 

\begin{lemma}\label{lemma hat F indep}
 Let $\rho$ be the frequency vector of $\M(\xi)$ and suppose that $\xi$ has irrational slope $\xi_2/\xi_1\in \R-\Q$. Then
 \[ \hat F (.,\rho) :\T\to\R , \qquad \frac{\partial \hat F}{\partial v}(.,\rho) :\T\to(\R^2)^* \]
 are constant.
\end{lemma}

\begin{proof}
 For the first claim note that the curves $x+t\rho$ are arc-length $\hat F$-geodesics, i.e.\ $\hat F(x,\rho)=1$ for all $x$. Now consider the Euler-Lagrange equation for $\hat F$, fulfilled by the $\hat F$-geodesics:
 \[ \frac{d}{dt} \frac{\partial \hat F}{\partial v}(c,\dot c) = \frac{\partial \hat F}{\partial x}(c,\dot c) . \]
 Then, for $c(t)=t\rho \mod \Z^2$ we find by $\hat F(.,\rho)=\const$, that
 \[ \frac{\partial \hat F}{\partial x}(c,\dot c) = \frac{\partial \hat F}{\partial x}(t\rho , \rho) = 0 . \]
 Hence, by the Euler-Lagrange equation
 \[ \frac{d}{dt} \frac{\partial \hat F}{\partial v}(t\rho,\rho) = 0 , \]
 such that $\frac{\partial \hat F}{\partial v}(t\rho,\rho)$ is independent of $t$. By the irrationality of $\xi$, the vector $\rho$ is also irrational (otherwise, $\M(\xi)$ would contain periodic orbits), such that the curve $c(t)= t\rho$ is dense in $\T$. The claim follows.
\end{proof}

In the next lemma, we obtain the desired estimates in Theorem \ref{main thm intro irrat} \eqref{main irrat i} for the Finsler metric $\hat F$. Recall the notation
\[ \beta_{\hat F} = \frac{1}{2} \sig_{\hat F}^2 . \]
Also note that $\beta_{\hat F}$ is differentiable in $\rho$ by Theorem \ref{thm intro mather-bangert} \eqref{thm intro mather-bangert irrat}.

\begin{lemma}\label{Taylor for stable norm}
 Let the regularity of the KAM-torus $\M(\xi)$ be $k\geq 3$ and let $\xi$ have irrational slope. Then there exists a constant $C\geq 1$, such that for all $h\in \R^2$
 \begin{align*}
  \frac{1}{C} |h-\rho|^2 \leq  \beta_{\hat F}(h) - \beta_{\hat F}(\rho) - D\beta_{\hat F}(\rho)[h-\rho] \leq C |h-\rho|^2 .
 \end{align*}
\end{lemma}

\begin{proof}
 We consider the Lagrangian $L=\frac{1}{2}\hat F^2$. Then by a Taylor expansion, for some $t_*\in (0,1)$
 \begin{align*}
  L(x,v) & = L(x,\rho) + \frac{\partial L}{\partial v}(x,\rho)[v-\rho] \\
  & \qquad \qquad + \frac{1}{2} \frac{\partial^2 L}{\partial v^2}(x,\rho+t_* (v-\rho))[(v-\rho),(v-\rho)] .
 \end{align*}
 As the Hessian $\frac{\partial^2 L}{\partial v^2}(x,\rho)$ is positive definite by the definition of a Finsler metric, we find a constant $C\geq 1$, such that
 \[ \forall (x,v) \in T\T, w\in \R^2 : \qquad \frac{1}{C} |w|^2 \leq \frac{1}{2} \frac{\partial^2 L}{\partial v^2}(x,v)[w,w] \leq C |w|^2 . \]
 Using Lemma \ref{lemma hat F indep}, the functions
 \begin{align*}
  & L(x,\rho) = \frac{1}{2}\hat F(x,\rho)^2 , \\
 & \frac{\partial L}{\partial v}(x,\rho)[w] = \frac{1}{2}\frac{d}{dt}\bigg|_{t=0} \hat F(x,\rho +tw)^2 = \hat F(x,\rho) \cdot \frac{\partial \hat F}{\partial v}(x,\rho)[w]
 \end{align*}
 are independent of $x\in\T$. Then consider the new Lagrangians
 \begin{align*}
  L_0^-(v) &= L(x,\rho) + \frac{\partial L}{\partial v}(x,\rho)[v-\rho] + \frac{1}{C} |v-\rho|^2 , \\
  L_0^+(v) &= L(x,\rho) + \frac{\partial L}{\partial v}(x,\rho)[v-\rho] + C |v-\rho|^2
 \end{align*}
 satisfying
 \begin{align}\label{L_0 approx L}
  L_0^-(v) \leq L(x,v) \leq L_0^+(v)
 \end{align}
 for all $(x,v)\in T\T$.
 
 We shall use \eqref{L_0 approx L} to find an analogous estimate for $\beta_{\hat F}$. For any curve $c:[0,T]\to \R^2$ we obtain by the $L^2$-Cauchy-Schwarz inequality, that
 \begin{align}\label{CS-ineq}
  l_{\hat F}(c;[0,T])^2 & = \la 1, \hat F(c,\dot c) \ra_{L^2}^2 \leq \|1\|_{L^2}^2 \cdot \|\hat F(c,\dot c)\|_{L^2}^2 \\
  \nonumber & = T \cdot \int_0^T \hat F(c,\dot c)^2 dt = 2T \cdot \int_0^T \hat L(c,\dot c) dt ,
 \end{align}
 with equality if and only if $\hat F(c,\dot c)=\const$. Let $z\in \Z^2$ and let $c:[0,T]\to\R^2$ for some $T>0$ be the shortest closed $\hat F$-geodesic in the homotopy class $z$ (i.e.\ $c(T)-c(0)=z$), parametrized to have constant $\hat F$-speed. Moreover, let $c_0:[0,T]\to\R^2$ be given by $c_0(t) = c(0) + t\frac{z}{T}$. If $L_0:T\T\to\R$ is any Tonelli Lagrangian independent of the base variable $x\in\T$, then fixing $T>0$ and the endpoints $c(0),c(T)\in \R^2$, it is known that the action $\int_0^T L_0(c,\dot c) dt$ is minimized by the curve $c_0$ with constant velocity vector (by Tonelli's Theorem there exist minimizers for fixed endpoints and fixed connection time, which have to be solutions of the Euler-Lagrange equation; by $L_0$ being independent of $x$, such minimizers have to be straight lines). It follows by \eqref{L_0 approx L} and \eqref{CS-ineq}, that
 \begin{align*}
  L_0^-(\tfrac{z}{T}) & = \frac{1}{T} \int_0^T L_0^-(c_0,\dot c_0) dt \leq \frac{1}{T} \int_0^T L_0^-(c,\dot c) dt \leq \frac{1}{T} \int_0^T L(c,\dot c) dt \\
  & = \frac{1}{2T^2} l_{\hat F}(c;[0,T])^2 \leq \frac{1}{2T^2} l_{\hat F}(c_0;[0,T])^2 \leq \frac{1}{T} \int_0^T L(c_0,\dot c_0) dt \\
  & \leq \frac{1}{T} \int_0^T L_0^+(c_0,\dot c_0) dt = L_0^+(\tfrac{z}{T}) .
 \end{align*}
 By $l_{\hat F}(c;[0,T]) = \sig_{\hat F}(z)$ and homogeneity of $\sig_{\hat F}$, we obtain
 \[ L_0^-(\tfrac{z}{T}) \leq \beta_{\hat F}(\tfrac{z}{T}) = \frac{1}{2} \sig_{\hat F}(\tfrac{z}{T})^2 \leq L_0^+(\tfrac{z}{T}) \qquad \forall z\in \Z^2, T>0 . \]
 Using continuity it follows that
 \begin{align*}
  L_0^-  \leq \beta_{\hat F} \leq L_0^+
 \end{align*}
 everywhere. Using the definition of $L_0^\pm$ one infers
 \begin{align*}
  L(x,\rho) = \beta_{\hat F}(\rho) , \qquad \frac{\partial L}{\partial v}(x,\rho) = D\beta_{\hat F}(\rho) 
 \end{align*}
 and the claim follows.
\end{proof}

We now translate Lemma \ref{Taylor for stable norm} into a statement for the original Finsler metric $F$ and obtain the main result in this section, which is stated as item \eqref{main irrat i} of Main Theorem \ref{main thm intro irrat} in the introduction.

\begin{thm}\label{thm KAM-torus}
 If $\xi\in \R^2-\{0\}$ has irrational slope and if the set $\M(\xi)$ is a $C^3$-KAM-torus for the geodesic flow $\phi_F^t$, then the square of the stable norm $\sig_F$ is strongly convex near $\xi$. More precisely, there exists a constant $C\geq 1$, such that for the function
 \[ \beta_F = \frac{1}{2} \sig_F^2 \]
 we have for all $h\in\R^2$
 \begin{align*}
  \frac{1}{C} \cdot| h- \xi|^2 \leq \beta_F( h) - \beta_F( \xi) - D\beta_F( \xi)[ h-\xi] \leq C \cdot | h-  \xi|^2 .
 \end{align*}
\end{thm}

\begin{proof}
 By Lemma \ref{Taylor for stable norm}, we find
 \begin{align*}
  \frac{1}{C} |h-\rho|^2 \leq \beta_{\hat F}(h) - \beta_{\hat F}(\rho) - D\beta_{\hat F}(\rho)[h-\rho] \leq  C |h-\rho|^2 .
 \end{align*}
 Using the first part of Lemma \ref{lemma vf on sigma},
 \begin{align*}
  & \beta_{\hat F} = \frac{1}{2} \sig_{\hat F}^2 = \frac{1}{2} (\sig_{F} \circ L_\vf^{-1})^2 = \beta_F\circ L_\vf^{-1} , \\
  & D\beta_{\hat F}(\rho) = D\beta_{F}(L_\vf^{-1}\rho) \circ L_\vf^{-1} .
 \end{align*}
 Also observe for the operator norm $\|.\|$, that
 \begin{align*}
  & |h-\rho|^2 \leq \|L_\vf\|^2 \cdot |L_\vf^{-1}h-L_\vf^{-1}\rho|^2 , \\
  & |L_\vf^{-1}h-L_\vf^{-1}\rho|^2  \leq \|L_\vf^{-1}\|^2 \cdot|h-\rho|^2 .
 \end{align*}
 Summarizing,
 \begin{align*}
  & \frac{1}{C\|L_\vf^{-1}\|^2} |L_\vf^{-1}h-L_\vf^{-1}\rho|^2 \\
  & \leq \beta_F(L_\vf^{-1}h) - \beta_F(L_\vf^{-1}\rho) - D\beta_{F}(L_\vf^{-1}\rho) [L_\vf^{-1}h-L_\vf^{-1}\rho] \\
  & \leq  C \|L_\vf\|^2 \cdot |L_\vf^{-1}h-L_\vf^{-1}\rho|^2 .
 \end{align*}
 The second part of Lemma \ref{lemma vf on sigma} showed $L_\vf^{-1}\rho = \lam \cdot \xi$ for some $\lam>0$. Replacing $C$ by a new constant and $L_\vf^{-1} h$ by $\lam \cdot h$ (as $h$ is arbitrary) yields
 \begin{align*}
  \frac{1}{C} |\lam \cdot h-\lam \cdot \xi|^2 & \leq \beta_F(\lam \cdot h) - \beta_F(\lam \cdot \xi) - D\beta_{F}(\lam \cdot \xi) [\lam \cdot h-\lam \cdot \xi] \\
  & \leq  C \cdot |\lam \cdot h-\lam \cdot \xi|^2 .
 \end{align*}
 Finally, using $\beta_F(\lam \cdot h) = \lam^2 \cdot \beta_F(h)$ and $D\beta_{F}(\lam \cdot \xi)= \lam \cdot D\beta_{F}(\xi)$ due to homogeneity and dividing the estimates by $\lam^2$, the claim follows.
\end{proof}

\section{The hyperbolic case}\label{section hyp}

In this section we study the stable norm $\sig=\sig_F$ of the Finsler metric $F$ near a direction $\xi\in S^1$ under the assumption, that the set $\M(\xi)\subset S\T$ of minimal geodesics with asymptotic direction $\xi$ carries hyperbolicity (see Definition \ref{def hyp intro}). We shall again write $\M(a\cdot \xi)=\M(\xi)$ for all $a>0$. The ideas in this section are motivated by the works of J.\ Mather, \cite{mather2} and \cite{mather1}.


\subsection{The rational case}\label{section rational}

Let us first recall the structure of the set $\M(\xi)$ for the case where $\xi\in S^1$ has rational or infinite slope. For simplicity, we will in this subsection rescale $\xi$:
\[ z := a \cdot \xi \in \Z^2 , \qquad a := \min \{t>0 : t\xi \in \Z^2 \} . \]
In particular $z$ is prime, i.e.\ there does not exist $w\in \Z^2$ and some $n\geq 2$ with $z=n\cdot w$.

The following theorem is due to H. M. Morse \cite{morse} and G. A. Hedlund \cite{hedlund}. The result has been generalized to non-reversible Finsler metrics by several authors, see e.g.\ \cite{zaustinsky}, \cite{carneiro} and \cite{paper1}. Let us write
\[ \M^{per}(z) := \{ v\in \M(z) : c_v \text{ is $z$-periodic}\}. \]
Here, a $z$-periodic geodesic is a closed geodesic in the homotopy class $z \in \Z^2\cong \pi_1(\T)$; equivalently, the lifts $\tilde c:\R\to\R^2$ are invariant under the translation by $z$.

\begin{thm}[Morse, Hedlund]\label{morse periodic}
 Let $\del\in S^1$ have rational or infinite slope and $z\in \Z^2$ as defined above. Then we have the following:
 \begin{enumerate}[(i)]
  \item \label{morse periodic item 1} $\M^{per}(z) \neq \emptyset$ and $\M^{per}(z)$ determines a closed lamination of $\T$, i.e.\ no two $z$-periodic minimal geodesics intersect. Moreover, if $z=kw$ for some $w\in\Z^2$ and $k\geq 2$, then the $z$-periodic minimal geodesics are the $k$-th iterates of the $w$-periodic minimal geodesics.
  
  In particular, we can say that a pair $c_0,c_1:\R\to \R^2$ of minimal geodesics lifted from $\M^{per}(z)$ is {\em neighboring}, if in the strip between $c_0(\R),c_1(\R) \subset\R^2$ there are no further geodesics lifted from $\M^{per}(z)$.
  
  
  \item \label{morse periodic item 3} If $c_0,c_1:\R\to \R^2$ is a pair of neighboring minimal geodesics lifted from $\M^{per}(z)$, then there exist two minimal geodesics $c^-,c^+:\R\to\R^2$ lifted from $\M(z)-\M^{per}(z)$, which are heteroclinic between $c_0,c_1$ with opposite asymptotic behavior (see Figure \ref{fig_structure-rational}).
 \end{enumerate}
\end{thm}

\begin{figure}\centering 
 \includegraphics[scale=1]{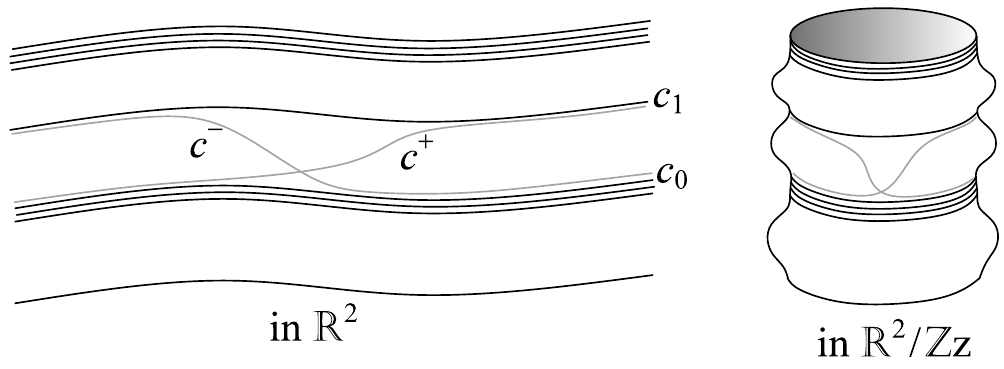}
 \caption{Some geodesics from $\M(z)$ for $z\in \Z^2-\{0\}$. There are some strips foliated by $z$-periodic geodesics and in between them there are gaps overstretched by heteroclinics. On the right one can see a geometric situation generating closed and heteroclinic minimals. \label{fig_structure-rational}}
\end{figure}

\begin{remark}
 It is known that $\M^{per}(z)$ consists precisely of the shortest closed geodesics in the homotopy class $z$. If the set $\M^{per}(z)$ is hyperbolic, then the heteroclinic connections between neighboring periodic minimals in Theorem \ref{morse periodic} \eqref{morse periodic item 3} are parts of the (un)stable manifolds of the periodic orbits.
\end{remark}

For the rest of this Section \ref{section hyp}, whenever we speak of a minimal geodesic, we mean a lift to the universal cover $\R^2$. We introduce some further notation.

\begin{defn}
 For $\xi=(\xi_1,\xi_2)\in \R^2$ we write $\xi^\perp := i\cdot \xi = (-\xi_2,\xi_1)$. We let $k(z)\in \N\cup\{\infty\}$ be the number of $z$-periodic minimals in $\R^2$ in the strip between a $z$-periodic minimal $c_0$ and its translate $c_0+z^\perp$.
\end{defn}

\begin{remark}\label{remark bounds k}
 \begin{enumerate}[(i)]
  \item In the case $k(z)<\infty$, we can choose any $z$-periodic minimal $c_0:\R\to\R^2$ and have the $z$-periodic minimals in the strip between $c_0,c_0+z^\perp$ ordered as
  \[ c_0 < c_1 < ... < c_{k(z)} = c_0+z^\perp , \]
  where we wrote $c_i<c_{i+1}$, if the image $c_{i+1}(\R)\subset\R^2$ lies left of $c_i(\R)$ with respect to the orientation given by the frame $(\dot c_i,(\dot c_i)^\perp)$.
 
  \item Assume that $\M^{per}(z)$ consists of a single orbit (this is conformally generic by Proposition \ref{generic-paper result}) and let $c_0:\R\to\R^2$ be a (lifted) $z$-periodic minimal. This means that in the strip between $c_0,c_0+z^\perp$, the $z$-periodic minimals correspond to the $(\Z^2+ c_0(0))$-points in that strip. The number $k(z)$ counts only the images of such minimals, so that using that $z$ is chosen to be prime one can easily show that
  \[ k(z) = 1 + \card \left(\Z^2 \cap Q(z) \right), \]
  where we wrote $Q(z)$ for the interior of the square spanned by $z,z^\perp$. Combining this with Pick's theorem, we find
  \[ k(z) = |z|\cdot |z^\perp| = |z|^2 , \]
  where $|.|$ is the euclidean norm on $\R^2$.
  
  \item If $\M^{per}(z)$ is assumed to be uniformly hyperbolic for the geodesic flow $\phi_F^t$, then $\M^{per}(z)\subset S\T$ is a finite union of closed orbits (by the lamination property in Theorem \ref{morse periodic}). Letting $n$ be the number of distinct orbits in $\M^{per}(z)$ we find
  \[ k(z) = n \cdot |z|^2 <\infty . \]
 \end{enumerate}
\end{remark}

The aim of this subsection is to prove Theorem \ref{intro rational hyp} from the introduction. For this we assume uniform hyperbolicity of the set $\M^{per}(z)$. As we noted after Theorem \ref{morse periodic}, the heteroclinic minimals $c^\pm$ between pairs of neighboring periodic minimals $c_0,c_1$ belong to the (un)stable manifolds. The hyperbolicity of $\M^{per}(z)$ will be used in the following form: There exist constants $C,\lam>0$ and parameters $S^\pm,T^\pm\in \R$, such that for all $t\geq 0$
\begin{align}\label{eq hyperbolicity}
 \begin{cases} d_F(c_1(-t),c^-(S^- -t)) &\leq C \exp(-\lam t) \\
 d_F(c^-(T^-+t),c_0(t)) &\leq C \exp(-\lam t) \\
 d_F(c_0(-t),c^+(S^+ -t)) &\leq C \exp(-\lam t) \\
 d_F(c^+(T^++t),c_1(t)) &\leq C \exp(-\lam t) 
 \end{cases} .
\end{align}
The notation will usually be that $c_1$ lies to the left of $c_0$ (using the orientation of the geodesics) and the choice of $c^-,c^+$ is such that the above inequalities hold.

The following lemma is the key observation in order to recognize hyperbolicity in the stable norm $\sig=\sig_F$ of $F$. Recall the definition
\begin{align*}
 D^+\sig(\xi)[v] &:=  \inf_{t>0} \frac{\sig(\xi+tv)-\sig(\xi)}{t} = \lim_{t\searrow 0} \frac{\sig(\xi+tv)-\sig(\xi)}{t} 
\end{align*}
and note that, due to homogeneity of $\sig$, for $a,b>0$
\[ D^+\sig(a \cdot \xi)[b \cdot v] = b \cdot D^+\sig(\xi)[v]. \]

\begin{lemma}\label{lemma stable norm nz pm z perp}
 Let $n\in\N$, $s\in \{-1,1\}$ and $z\in \Z^2-\{0\}$, such that $\M^{per}(z)$ is a hyperbolic set for the geodesic flow with hyperbolicity constants $C,\lam$ (in the sense of \eqref{eq hyperbolicity}). Then for integers $n_i$ with
 \[ 0 = n_0 \leq n_1 \leq ... \leq n_{k(z)-1} \leq n_{k(z)} = n \]
 there exists a continuous, piecewise $C^1$, $(nz +s z^\perp)$-periodic curve $\g$ with
 \begin{align*}
  &\l_F(\g) \leq \sig(nz) + D^+\sig(z)[s z^\perp] + 2C \sum_{i=1}^{k(z)} \exp\left(-\frac{\lam \sig(z)}{2} (n_i-n_{i-1}) \right) . 
 \end{align*}
\end{lemma}

\begin{proof}
 We first consider the case $s=1$. Set $k=k(z)$, let $c_0,...,c_{k-1}$ be the ordered sequence of periodic minimals between $c_0$ and $c_k := c_0+z^\perp$. Consider heteroclinic minimals $c_i^+$ connecting $c_{i-1}$ to $c_i$ for $i=1,...,k$. By hyperbolicity, we find $C,\lam>0$ and $T_1,...,T_k\in\R$ with
 \begin{align}\label{using hyp in lemma}
  \begin{cases} d_F(c_{i-1}(-t),c_i^+(-t)) & \leq C \exp(-\lam t) \\
  d_F(c_i^+(T_i+t),c_i(t)) & \leq C \exp(-\lam t) \end{cases}
 \end{align}
 for all $t\geq 0$. Here we change the parameter of the $c_i^+$ if necessary to have $T_i$ only in the second line. Set $\theta=\sig(z)$ and given the integers $n_i$, define
 \[ S_i := \theta \cdot \frac{n_{i}-n_{i-1}}{2} \geq 0, \quad i=1,...,k , \qquad S_0:=S_k . \]
 Consider minimal geodesic segments $\delta_i,\e_i$ for $i=1,...,k$, connecting
 \[ \delta_i : c_{i-1}(-S_{i-1}) \to c_i^+(-S_{i-1}) , \qquad \e_i : c_i^+(T_i+S_i) \to c_i(S_i) , \]
 and set
 \begin{align*}
  \g := & ~ [~ n_{0} \cdot z + (\delta_1 * c_1^+|_{[-S_0,T_1+S_1]} * \e_1) ~] \\
  * & ~ ... \\
  * & ~ [~ n_{i-1} \cdot z + (\delta_i * c_i^+|_{[-S_{i-1},T_i+S_i]} * \e_i) ~] \\
  * & ~ ... \\
  * & ~ [ ~ n_{k-1} \cdot z + (\delta_k * c_k^+|_{[-S_{k-1},T_k+S_k]} * \e_k) ~] .
 \end{align*}
 Here we wrote $*$ for the concatenation of curves. By definition of $S_i$ and $\theta$ being the period of $c_i$ we have for $i=1,...,k-1$, that
 \[ c_i(S_i)-c_i(-S_i) = (n_i-n_{i-1})\cdot z , \]
 showing
 \[ n_{i-1} \cdot z + c_i(S_i) = n_i \cdot z + c_i(-S_i) . \]
 Hence, the curve $\g$ is a continuous, piecewise $C^1$-connection from $c_{0}(-S_0)$ to $n_{k-1} \cdot z + c_k(S_k)$. Moreover, by $c_k=c_0+z^\perp$ and $S_0=S_k$ we have
 \begin{align*}
  n_{k-1} \cdot z + c_k(S_k) & = n_{k-1} \cdot z + z^\perp + c_0(S_0) \\
  & = n_{k-1} \cdot z + z^\perp + c_0(\theta \cdot(n_k-n_{k-1}) -S_0) \\
  & = n_k \cdot z + z^\perp + c_0(-S_0) ,
 \end{align*}
 i.e.\ with $n_k=n$ the curve $\g$ is $(n z + z^\perp)$-periodic. Using the hyperbolicity and $S_k=S_0$, we find
 \begin{align}\label{formula hyp in the proof}
  \nonumber \sig(nz+z^\perp) & \leq l_F(\g) = \sum_{i=1}^k \bigg[ T_i+S_i+S_{i-1} + l_F(\delta_i) + l_F(\e_i) \bigg] \\
  & \leq \sum_{i=1}^k T_i + 2 \sum_{i=1}^k S_i + C \sum_{i=1}^k \left(\exp(-\lam S_{i-1})+\exp(-\lam S_i)\right) \\
  \nonumber & = \sum_{i=1}^k T_i + n\theta + 2C \sum_{i=1}^k  \exp\bigg(-\lam \theta \frac{n_{i}-n_{i-1}}{2} \bigg) . 
 \end{align}
 Note that $n\theta=n\sig(z)=\sig(nz)$. To finish the proof we show
 \[ \sum_{i=1}^k T_i \leq D^+\sig(z)[z^\perp] . \]
 
 Let $\g_n$ be a $(nz+z^\perp)$-periodic minimal. We find intersections (the $c_0,...,c_{k-1},c_k=c_0+z^\perp$ as before)
 \begin{align*}
  & \g_n(S_{n,i}) \in c_i(\R) 
 \end{align*}
 and w.l.o.g. we have
 \[ 0 = S_{n,0} <S_{n,1} < ... < S_{n,k} = \sig(nz+z^\perp) \]
 and
 \[ \g_n(S_{n,i}) = c_i(0) ~ \forall  i = 1,...,k-1 , \qquad \g_n(S_{n,k}) = c_k(n\theta). \]
 Define for $i=1,...,k$
 \[ \g_n^i(t) := \begin{cases}
              c_{i-1}(t) & : t\leq 0 \\
              \g_n(S_{n,i-1} + t) & : 0 \leq t \leq S_{n,i}-S_{n,i-1} \\
              c_i(t - [S_{n,i}-S_{n,i-1}]) & : t \geq S_{n,i}-S_{n,i-1} ~\&~ i < k \\
              c_i(t - [S_{n,i}-S_{n,i-1}] + n\theta ) & : t \geq S_{n,i}-S_{n,i-1} ~\&~ i = k
             \end{cases}  . \]
 The $\g_n^i$ are heteroclinic curves connecting $c_{i-1}$ to $c_i$. For large $m$, the curve $c_i^+|_{[-m\theta, m\theta + T_i]}$ connects by \eqref{using hyp in lemma} approximately the point
 \[ c_{i-1}(-m\theta) = \g_n^i(-m\theta) \]
 to the point
 \[ c_i(m\theta) = \begin{cases}
                                                                         \g_n^i(m\theta + [S_{n,i}-S_{n,i-1}]) & : i < k \\
                                                                         \g_n^i((m-n)\theta + [S_{n,i}-S_{n,i-1}]) & : i = k
                                                                        \end{cases} . \]
 By the minimality of the curves $c_i^+$ we find for all $n$
 \begin{align*}
  0 & \leq \lim_{m\to\infty} \begin{pmatrix} \sum_{i=1}^{k-1} l_F(\g_n^i;[-m\theta, m\theta + S_{n,i}-S_{n,i-1}]) \\
  + l_F(\g_n^k;[-m\theta, (m-n)\theta + S_{n,k}-S_{n,k-1} ]) \\
  - \sum_{i=1}^k l_F(c_i^+;[-m\theta, m\theta + T_i]) \end{pmatrix}\\
  & = \lim_{m\to\infty} \begin{pmatrix} \sum_{i=1}^{k-1} 2m\theta + S_{n,i}-S_{n,i-1} \\
  + (2m-n)\theta + S_{n,k}-S_{n,k-1} \\
  - \sum_{i=1}^k 2m\theta + T_i \end{pmatrix} \\
  & = \lim_{m\to\infty} S_{n,k}-S_{n,0}  - n\theta - \sum_{i=1}^k T_i \\
  & = \sig(nz+z^\perp)-\sig(nz) - \sum_{i=1}^k T_i .
 \end{align*}
 The claim follows, observing that
 \begin{align}\label{formula d+sig mit n}
  D^+\sig(z)[z^\perp] = \lim_{n\to\infty} \frac{\sig(z+\frac{1}{n}z^\perp)-\sig(z)}{\frac{1}{n}} = \lim_{n\to\infty} \sig(nz+ z^\perp)-\sig(nz) . 
 \end{align}
 
 The estimates for a $(nz-z^\perp)$-periodic curve, i.e.\ the case $s=-1$ follow by the same lines just using the heteroclinics $c_i^-$.
\end{proof}

\begin{remark}
 In the proof of Lemma \ref{lemma stable norm nz pm z perp} we showed $\sum T_i \leq D^+\sig(z)[z^\perp]$. The reverse inequality also holds, which can be seen using the curve $\g$ associated to the integers $n_i = i \cdot \lfloor n/k\rfloor$ for $i<k$ and then applying the estimate \eqref{formula hyp in the proof} to the formula \eqref{formula d+sig mit n}.
\end{remark}

We need to refine Lemma \ref{lemma stable norm nz pm z perp}.

\begin{lemma}\label{lemma stable norm nz + m z perp}
 Let $a\in [-1,1]$ and $z\in \Z^2-\{0\}$, such that $\M^{per}(z)$ is a hyperbolic set with hyperbolicity constants $C,\lam$ in \eqref{eq hyperbolicity}. Then with the Gaussian bracket $\lfloor x \rfloor = \max\{n\in\Z : n\leq x\}$ 
 \begin{align*}
  & \sig(z+ a z^\perp) \leq \sig(z) + D^+\sig(z)[a z^\perp] + 2C |a| k(z) \exp\left(-\frac{\lam \sig(z)}{2} \left\lfloor \frac{1}{|a|k(z)} \right\rfloor \right) .
 \end{align*}
\end{lemma}

\begin{proof}
 Consider integers $N\geq M\geq 1$ and write $k=k(z)$. We choose integers $n_i, n_i^*$ defined as
 \begin{align*}
 & n_i := i\cdot \left\lfloor \frac{N}{Mk} \right\rfloor , \qquad i = 0,...,k ,&& \\
 & n_i^*:= n_i , \qquad i = 0,...,k-1, && n_k^* := N-(M-1)\cdot n_k .
 \end{align*}
 Observe that by $N\geq M$
 \[ n_k^* = N-(M-1)\cdot k\cdot \left\lfloor \frac{N}{Mk} \right\rfloor \geq k\cdot \left\lfloor \frac{N}{Mk} \right\rfloor = n_{k-1}^* + \left\lfloor \frac{N}{Mk} \right\rfloor . \]
 For $s\in \{-1,1\}$ let $\g$ be the $(n_kz+sz^\perp)$-periodic curve from Lemma \ref{lemma stable norm nz pm z perp} associated to the integers $n_i$ and analogously $\g^*$ the $(n_k^*z+sz^\perp)$-periodic curve associated to the integers $n_i^*$. Consider the new curve
 \begin{align*}
  \Gamma := ~ & \g * (\g + [n_k z + sz^\perp]) * ...  * (\g + (M-2)[n_k z + sz^\perp]) \\
   & \qquad \qquad\qquad\qquad\qquad\qquad\qquad * (\g^* + (M-1)[n_k z + sz^\perp]) .
 \end{align*}
 We find for the homotopy class
 \[ [\Gamma] = (M-1) (n_kz+sz^\perp) + (n_k^*z+sz^\perp) =Nz+sMz^\perp . \]
 Note that $n_i-n_{i-1}, n_i^*-n_{i-1}^* \geq \left\lfloor \frac{N}{Mk} \right\rfloor$, so that with Lemma \ref{lemma stable norm nz pm z perp}
 \begin{align*}
  & \sig(Nz+ sMz^\perp) \\
  & \leq l_F(\Gamma) = (M-1) l_F(\g) + l_F(\g^*) \\
  & \leq (M-1)\sig(n_kz) + \sig(n_k^*z) + M D^+\sig(z)[s z^\perp] \\
  & \qquad\qquad\qquad\qquad\qquad\qquad\qquad + 2CM k \exp\left(-\frac{\lam \sig(z)}{2} \left\lfloor \frac{N}{Mk} \right\rfloor \right) \\
  & = \sig(Nz) + M D^+\sig(z)[s z^\perp] + 2CM k \exp\left(-\frac{\lam \sig(z)}{2} \left\lfloor \frac{N}{Mk} \right\rfloor \right) .
 \end{align*}

 Now let $a \in [-1,1]-\{0\}$ be arbitrary (note that the lemma is trivial for $a=0$). Let $s=\sign(a)$ and choose sequences of integers $N_n\geq M_n\geq 1$ with $M_n/N_n \nearrow |a|$, so that
 \[ (1,|a|) = \lim_{n\to\infty} \frac{1}{N_n} (N_n,M_n) , \qquad \frac{N_n}{M_n}\geq \frac{1}{|a|}.\]
 We find by the continuity of $\sig$, monotonicity of $\lfloor . \rfloor$ and homogeneity of $D^+\sig(z)[.]$
 \begin{align*}
  & \sig(z+ a z^\perp) = \lim_{n\to\infty} \frac{1}{N_n}  \sig( N_n z + s M_n z^\perp) \\
  & \leq \lim_{n\to\infty} \frac{1}{N_n} \left[ \sig(N_nz) + M_n  D^+\sig(z)[s z^\perp] + 2CM_n k \exp\left(-\frac{\lam \sig(z)}{2} \left\lfloor \frac{1}{|a| k} \right\rfloor \right) \right] \\
  & = \sig(z) + D^+\sig(z)[s |a| z^\perp] + 2C |a| k \exp\left(-\frac{\lam \sig(z)}{2} \left\lfloor \frac{1}{|a|k} \right\rfloor \right) .
 \end{align*}
 The lemma follows.
\end{proof}

Let us prove a simple lemma, which will be useful again later.

\begin{lemma}\label{lemma diff sig radially}
 Fix $\xi\in \R^2-\{0\}$ and consider the maps
 \[ \Pi_1, \Pi_2 : \R^2 \to \R^2 , \quad \Pi_1(v) := (|\xi|+\la v,\tfrac{\xi}{|\xi|} \ra) \tfrac{\xi}{|\xi|} , \quad \Pi_2(v) := v - \la v,\tfrac{\xi}{|\xi|} \ra \tfrac{\xi}{|\xi|} . \]
 Then 
 \begin{enumerate}[(i)]
  \item\label{lemma diff sig radially i} $\xi + v = \Pi_1(v) + \Pi_2(v)$ and $\Pi_1(v) \perp \Pi_2(v)$,
  
  \item\label{lemma diff sig radially ii} for $|v|\leq|\xi|$ we have
   \begin{align*}
    |\Pi_1(v)| & = |\xi|+ \la v,\tfrac{\xi}{|\xi|} \ra , \qquad |\Pi_2(v)| = \left| \la v,\tfrac{\xi^\perp}{|\xi|} \ra \right| .
   \end{align*}
   
  \item\label{lemma diff sig radially iii} for $|v|<|\xi|$ we have
    \begin{align*}
     & \sig(\xi+v) - \sig(\xi) - D^+\sig(\xi)[v] \\
     & = \sig(\Pi_1(v) + \Pi_2(v)) - \sig(\Pi_1(v)) - D^+\sig(\Pi_1(v))[\Pi_2(v)] .
    \end{align*}
 \end{enumerate}
\end{lemma}

\begin{proof}
 Items \eqref{lemma diff sig radially i} and \eqref{lemma diff sig radially ii} follow directly from the definitions. Let us prove \eqref{lemma diff sig radially iii}. Using homogeneity of $\sig$ we find for $a\in\R$
 \begin{align*}
  D^+\sig(\xi)[a\xi+v] & = \lim_{t\searrow 0}\frac{\sig((1+ta)\xi+tv)-\sig(\xi)}{t} \\
  & = \lim_{t\searrow 0}\frac{(1+ta)\sig(\xi+\frac{t}{1+ta}v)- (1+ta)\sig(\xi) +ta\sig(\xi)}{t} \\
  & = \lim_{t\searrow 0}\frac{\sig(\xi+\frac{t}{1+ta}v)- \sig(\xi) }{\frac{t}{1+ta}} +a\sig(\xi) \\
  & = a\sig(\xi) + D^+\sig(\xi)[v] .
 \end{align*}
 Then using $|v|<|\xi|$ and homogeneity it follows that
 \begin{align*}
  & \sig(\xi+v) - \sig(\xi) - D^+\sig(\xi)[v] \\
  & = \sig(\Pi_1(v) + \Pi_2(v)) - \sig(\xi) - \left(\la v,\tfrac{\xi}{|\xi|^2} \ra \sig(\xi) + D^+\sig(\xi)[\Pi_2(v)]\right) \\
  & = \sig(\Pi_1(v) + \Pi_2(v)) - \sig(\Pi_1(v)) - D^+\sig(\Pi_1(v))[\Pi_2(v)] .
 \end{align*}
\end{proof}

We can now prove the main result of Subsection \ref{section rational}.

\begin{thm}\label{thm rational}
 Let $\xi,v\in \R^2$ with $\xi\neq 0$ having rational or infinite slope, such that $\M^{per}(\xi)$ is uniformly hyperbolic. Moreover, let $z$ be the prime element in $\Z^2\cap \R_{>0}\xi$ and $C,\lam$ the hyperbolicity constants from \eqref{eq hyperbolicity}. Then for all $v\in \R^2$ with $\left|\la v,\tfrac{\xi^\perp}{|\xi|} \ra\right| < |\xi| +\la v, \tfrac{\xi}{|\xi|} \ra$
 \begin{align*}
  & \sig(\xi+v) - \sig(\xi) - D^+\sig(\xi)[v] \\
  & \leq \left|\la v,\tfrac{\xi^\perp}{|\xi|} \ra\right| \cdot \frac{k(z)}{|z|} \cdot 2C \cdot \exp\left(-\frac{\lam \sig(z)}{2} \cdot \left\lfloor \frac{ |\xi| +\la v, \tfrac{\xi}{|\xi|} \ra }{\left|\la v,\tfrac{\xi^\perp}{|\xi|} \ra\right| \cdot k(z)}  \right\rfloor \right) .
 \end{align*}
\end{thm}

\begin{proof}
 We first prove the theorem for $v\perp \xi$ and choose some $a>0$ and $b\in \R$ with
 \[ \xi = a \cdot z , \qquad v = b \cdot z^\perp . \]
 Then by Lemma \ref{lemma stable norm nz + m z perp}, $D^+\sig(z)[.]=D^+\sig(\xi)[.]$ and $a=|\xi|/|z| > |b|=|v|/|z|$
 \begin{align*}
  \sig(\xi+v) & = a \cdot \sig \left( z + \frac{b}{a}z^\perp \right) \\
  & \leq a \bigg[ \sig\left( z \right) + D^+\sig(z)\left[ \frac{b}{a} z^\perp \right] \\
  & \qquad\qquad\qquad\qquad\qquad + 2C  \frac{|b|}{a} k(z) \cdot \exp\left(-\frac{\lam \sig(z)}{2} \left\lfloor \frac{a}{|b| k(z)}  \right\rfloor \right) \bigg] \\
  & = \sig\left( \xi \right) + D^+\sig(\xi)\left[ v \right] + 2C |v|\frac{k(z)}{|z|}   \cdot \exp\left(-\frac{\lam \sig(z)}{2} \left\lfloor \frac{|\xi|}{|v| k(z)}  \right\rfloor \right) .
 \end{align*}

 Let now $v\in\R^2$ with $\left|\la v,\tfrac{\xi^\perp}{|\xi|} \ra\right| = |\Pi_2(v)| < |\xi| +\la v, \tfrac{\xi}{|\xi|} \ra = |\Pi_1(v)|$. We can reduce this case to the case $v\perp\xi$ by using Lemma \ref{lemma diff sig radially}. Namely,
 \begin{align*}
  & \sig(\xi+v) - \sig(\xi) - D^+\sig(\xi)[v] \\
  & = \sig(\Pi_1(v) + \Pi_2(v)) - \sig(\Pi_1(v)) - D^+\sig(\Pi_1(v))[\Pi_2(v)] \\
  & \leq 2C|\Pi_2(v)|\frac{k(z)}{|z|} \cdot \exp\left(-\frac{\lam \sig(z)}{2} \left\lfloor \frac{|\Pi_1(v)|}{|\Pi_2(v)| k(z)}  \right\rfloor \right) .
 \end{align*}
 This proves the theorem using the formulae in Lemma \ref{lemma diff sig radially}.
\end{proof}

We can now easily prove Theorem \ref{intro rational hyp} from the introduction.

\begin{proof}[Proof of Theorem \ref{intro rational hyp}]
 Recall that
 \[ x/2\leq \lfloor x\rfloor \qquad \forall x\geq 1 . \]
 Also note that $\left|\la v,\tfrac{\xi^\perp}{|\xi|} \ra \right| \leq |v|$. With $|\xi|+\la v, \tfrac{\xi}{|\xi|}\ra \geq \frac{|\xi|}{2}$ for $|v|\leq |\xi|/2$ we find by the monotonicity of $\lfloor . \rfloor$
 \begin{align*}
  \left\lfloor \frac{ |\xi| +\la v, \tfrac{\xi}{|\xi|} \ra }{\left|\la v,\tfrac{\xi^\perp}{|\xi|} \ra\right| k(z)}  \right\rfloor \geq \left\lfloor \frac{ |\xi|}{2 |v|  k(z)} \right\rfloor \geq \frac{ |\xi|}{4 |v|  k(z)}
 \end{align*}
 for $\frac{ |\xi|}{2 |v|  k(z)} \geq 1$, or equivalently $|v| \leq \frac{ |\xi|}{2 k(z)}$. We apply Theorem \ref{thm rational} and find
 \begin{align*}
  \sig(\xi+v) - \sig(\xi) - D^+\sig(\xi)[v] & \leq  |v| 2C\frac{k(z)}{|z|} \exp\left(-\frac{\lam \sig(z) |\xi|}{8 k(z)} \frac{ 1}{|v| } \right) .
 \end{align*}
 The claim follows.
\end{proof}

\subsection{The irrational case}

In order to see hyperbolicity in irrational directions using the stable norm, we will approximate irrationals by rationals and then use the results of the previous Subsection \ref{section rational}. We shall use the following well-known continuity property of the sets $\M(\xi)$, based on the continuity of the asymptotic direction $\delta^+: \cup_{\xi\in S^1}\M(\xi)\to S^1$ and on the closedness of the minimality-condition. We omit the proof; see e.g.\ Corollary 3.16 in \cite{bangert}.

\begin{lemma}\label{M upper semi-cont}
 Let $\xi_n\to\xi$ in $S^1$ and consider any sequence $v_n\in \M(\xi_n)$. Then any limit point of $\{v_n\}$ lies in $\M(\xi)$. Equivalently, if $\xi_n\to\xi$, then for any open neighborhood $U\subset S\T$ of $\M(\xi)$, there exists $n_0\in \N$, such that $\M(\xi_n)\subset U$ for $n\geq n_0$.
\end{lemma}

The following theorem is item \eqref{main irrat ii} of Theorem \ref{main thm intro irrat} from the introduction. In the proof, we recover Theorem \ref{thm rational} in the irrational case.

\begin{thm}\label{slope-lemma}
 Let $\xi \in \R^2-\{0\}$ have irrational slope $\xi_2/\xi_1\in \R - \Q$ and assume that the set $\M(\xi)$ is uniformly hyperbolic. Assume moreover that $k(z)=|z|^2$ for all $z\in \Z^2-\{0\}$ (see Remark \ref{remark bounds k}). Then there exist constants $C,\lam>0$, such that for all choices of rays $R \subset \R^2$ emanating from the origin (i.e.\ $R=\R_{>0} \cdot v$ for some $v\neq 0$) there exist sequences $v_n \in R$ with $v_n\to 0$, so that
 \begin{align*}
  & \sig(\xi+v_n) - \sig(\xi) - D^+\sig(\xi)[v_n] \leq |v_n|^{1/4} \cdot C \cdot \exp\left(-\lam \cdot \frac{1}{| v_n |^{1/4}} \right) .
 \end{align*}
\end{thm}


\begin{proof}
 By Proposition 6.4.6 on p.\ 265 in \cite{KH}, there exists an open neighborhood $U\supset\M(\xi)$, such that the set
 \[ \Lambda := \bigcap_{t\in\R} \phi_F^t(\overline U) \]
 is uniformly hyperbolic for $\phi_F^t$. If we approach $\xi$ by a sequence $\xi_n\in S^1$ with rational slopes, then for large $n$ the sets $\M(\xi_n)$ will lie in the neighborhood $U$ using the upper semi-continuity of $\del\mapsto\M(\del)$ in Lemma \ref{M upper semi-cont}. By the flow-invariance of these sets, all sets $\M(\xi_n)$ will lie in $\Lambda$ and hence be hyperbolic with the hyperbolicity constants $C,\lam>0$ of $\Lambda$ (for $n$ sufficiently large). We will use the hyperbolicity of such $\M(\xi_n)$ in the sense of \eqref{eq hyperbolicity}.

 We first assume $\xi\perp R$ in the euclidean sense. We approximate $\xi$ by a point $\eta \in \xi - R$ of rational slope $r\in \Q$. More precisely, using continued fractions one can find infinitely many approximations $r :=p/q \in \Q$ of the slope $\om := \xi_2/\xi_1\in \R-\Q$ of $\xi$ satisfying 
 \begin{align}\label{estimate diophantine}
  |\om-r|\leq \frac{1}{q(q+1)} ,
 \end{align}
 such $r$'s lying on either side of $\om$. Thus, the set 
 \[ Q=Q(\xi,R) := \{ q\in \N ~|~ \exists p \in \Z : \eta\in \xi-R \text{ has slope } r=p/q \text{ with \eqref{estimate diophantine}}\} \]
 is unbounded.

 We fix $q\in Q$ and estimate the distance $|\xi-\eta|$. Using that $\eta$ is the only point in the straight line $\xi+\Span(R)$ of slope $\eta_2/\eta_1=r$, one verifies that
 \begin{align*}
  \eta = \xi + \frac{r\xi_1-\xi_2}{\xi_1+r\xi_2} \cdot \xi^\perp = \xi + \frac{r-\om}{1+r\om} \cdot \xi^\perp .
 \end{align*}
 This shows with $r\om>0$ by $r\approx \om$ and \eqref{estimate diophantine}
 \begin{align}\label{formula xi-eta}
  \frac{|\xi-\eta|}{|\xi|} &= \frac{|\om-r|}{1+ r\om} \leq \frac{1}{q(q+1)} \cdot \frac{1}{1+ r\om}.
 \end{align}

 Next, we let $v\in R$ with
 \begin{align}\label{cond v}
  \frac{1}{2q^4} \leq \frac{|v|}{|\xi|} \leq \frac{1}{q^4} .
 \end{align}
 For the intuition observe that using the orientation given by $R$, the points $\eta,\xi,\xi+v$ are ordered along the line $\xi+\Span(R)$ as
 \begin{align}\label{eq orientation}
  \eta < \xi < \xi+v .
 \end{align}

 We shall prove the theorem in the case $\xi \perp R$ in two steps. In the first step, we estimate the Gaussian bracket appearing in our application of Theorem \ref{thm rational} from below by $1$. As the proof of Step 1 is quite technical and long, we delay it to the end of this proof of Theorem \ref{slope-lemma}.

 \abs
 
 {\em Step 1.}
 Assuming $\xi\perp R$, with $\eta, v$ satisfying \eqref{formula xi-eta}, \eqref{cond v}, $z$ being the prime element in $\Z^2\cap \R_{>0}\eta$ and $t := \frac{|\xi-\eta|}{|v|}$ we have for sufficiently large $q\in Q$, that
 \begin{align}\label{cond gaussian}
   & \frac{ |\eta| +\la (1+t)v, \tfrac{\eta}{|\eta|} \ra }{\left|\la (1+t)v ,\tfrac{\eta^\perp}{|\eta|} \ra\right| \cdot k(z)} \quad \geq \quad 1 . 
 \end{align}

 \abs
 
 {\em Step 2.}
 Assuming $\xi\perp R$, there exist constants $C',\lam'>0$ (depending only on $\xi/|\xi|\in S^1$), so that for sufficiently large $q\in Q$ and $v\in R$ satisfying \eqref{cond v} we have
 \begin{align}\label{thm for L perp xi}
  & \sig(\xi+v) - \sig(\xi) - D^+\sig(\xi)[v]  \leq |\xi|^{3/4} \cdot|v|^{1/4} \cdot C' \cdot \exp\left(-\lam' \cdot \frac{|\xi|^{1/4}}{|v|^{1/4}} \right) .
 \end{align}
 
 \begin{proof}[Proof of Step 2] 
 We let $t = \frac{|\xi-\eta|}{|v|}>0$ as above, so that $\xi=\eta+tv$ (recall \eqref{eq orientation}). By definition of $D^+\sig$ one finds
 \[ \sig(\xi) = \sig(\eta+tv) \geq \sig(\eta) + t \cdot D^+\sig(\eta)[v] . \] 
 Moreover, by convexity of $\sig$ and \eqref{eq orientation} we find
 \[ D^+\sig(\eta)[v] = \inf_{s>0} \frac{\sig(\eta+sv)-\sig(\eta)}{s} \leq \inf_{s>0} \frac{\sig(\xi+sv)-\sig(\xi)}{s} = D^+\sig(\xi)[v] . \]
 Hence,
 \begin{align}\label{sig xi to sig eta}
  & \nonumber \sig(\xi+v) - \sig(\xi) - D^+\sig(\xi)[v] \\
  & \leq \sig(\xi+v) - \sig(\eta) - t \cdot D^+\sig(\eta)[v] - D^+\sig(\eta)[v] \\
  & \nonumber = \sig(\eta+(1+t)v) - \sig(\eta) - D^+\sig(\eta)[(1+t)v] .
 \end{align}
 Our aim is to apply Theorem \ref{thm rational} to the rational direction $\eta$, the vector $(1+t)v$ and $z$ the prime element in $\Z^2\cap \R_{>0}\eta$. Choose some $\al>0$ with $\al|.|\leq F$ and note that
 \[ \sig \geq \al |.| .\]
 Choo\-sing $D> 1+\om^2$, we find for large $q$
 \[ q^2\leq |z|^2=p^2+q^2=q^2(1+r^2) \leq D q^2. \]
 Due to $t = \frac{|\xi-\eta|}{|v|}$, \eqref{formula xi-eta} and \eqref{cond v} we have
 \[ \left|\la (1+t)v,\tfrac{\eta^\perp}{|\eta|} \ra\right| \leq |(1+t)v| = |v| + |\xi-\eta| \leq \frac{|\xi|}{q^4} + \frac{1}{q(q+1)} \cdot \frac{|\xi|}{1+ r\om} \leq \frac{2|\xi|}{q^2} .\]
 Also, \eqref{cond v} implies
 \[ \frac{1}{q} \leq 2 \frac{|v|^{1/4}}{|\xi|^{1/4}} . \]
 We now apply Theorem \ref{thm rational} to \eqref{sig xi to sig eta} obtaining by Step 1 and $k(z)=|z|^2$
 \begin{align*}
  & \sig(\xi+v) - \sig(\xi) - D^+\sig(\xi)[v] \\
  & \leq \left|\la (1+t)v,\tfrac{\eta^\perp}{|\eta|} \ra\right| \cdot \frac{k(z)}{|z|} \cdot 2C \cdot \exp\left(-\frac{\lam \sig(z)}{2} \cdot \left\lfloor \frac{ |\eta| +\la (1+t) v, \tfrac{\eta}{|\eta|} \ra }{\left|\la (1+t) v,\tfrac{\eta^\perp}{|\eta|} \ra\right| \cdot k(z)}  \right\rfloor \right) \\
  & \leq \frac{|\xi|}{q} \cdot 4C\sqrt D \cdot \exp\left(-\frac{\lam \al }{2} \cdot q \right) \\
  & \leq |\xi|^{3/4} \cdot |v|^{1/4} \cdot 8C\sqrt D \cdot \exp\left(-\frac{\lam \al }{4} \cdot \frac{|\xi|^{1/4}}{|v|^{1/4}} \right) .
 \end{align*}
 Observe that the condition on the vector $(1+t)v$ in Theorem \ref{thm rational} is satisfied due to Step 1.
 \end{proof}

 Next, we generalize our estimate from Step 2 to the general case, not assuming $\xi\perp R$. Observe that the case $R \subset \R\xi$ is trivial, so we assume $R\not\subset\R\xi$. Using Lemma \ref{lemma diff sig radially} and Step 2, we find
 \begin{align*}
  & \sig(\xi+v) - \sig(\xi) - D^+\sig(\xi)[v] \\
  & = \sig(\Pi_1(v) + \Pi_2(v)) - \sig(\Pi_1(v)) - D^+\sig(\Pi_1(v))[\Pi_2(v)] \\
  & \leq |\Pi_1(v)|^{3/4} \cdot|\Pi_2(v)|^{1/4} \cdot C' \cdot \exp\left(-\lam' \cdot \frac{|\Pi_1(v)|^{1/4}}{|\Pi_2(v)|^{1/4}} \right) .
 \end{align*}
 For the above argument we need to meet the requirement \eqref{cond v} for Step 2, i.e.\ we need
 \begin{align}\label{cond v neu}
  \frac{1}{2q^4} \leq \frac{|\Pi_2(v)|}{|\Pi_1(v)|} \leq \frac{1}{q^4} .
 \end{align}
 Note that the set $Q$ of possible $q$ depends only the choice of $R$ and on $\om=\xi_2/\xi_1$, not on $\xi$ itself. Choosing $v_0\in R$ with $|v_0|=1$, observe that by Lemma \ref{lemma diff sig radially}
 \begin{align*}
  |\Pi_1(v)| &= |\xi|+|v| \cdot \la v_0,\tfrac{\xi}{|\xi|} \ra , \qquad |\Pi_2(v)| = |v| \cdot\left| \la v_0,\tfrac{\xi^\perp}{|\xi|} \ra \right| .
 \end{align*}
 Both terms do not vanish (using $|v|<|\xi|$ and $R\not\subset\R\xi$). Note that for $a,b>0,c \in \R$ the function $x\mapsto \frac{a x}{b+c x}$ is monotone and bijective near $x=0$. Hence, the range for $\frac{|\Pi_2(v)|}{|\Pi_1(v)|}$ in \eqref{cond v neu} translates in a 1-1 fashion into a range for $|v|$. The theorem follows, using for $|v|\leq |\xi|/2$ the estimates
 \begin{align*}
  \tfrac{1}{2}|\xi| \leq |\xi|-|v| \leq |\Pi_1(v)| \leq 2|\xi| , \qquad |\Pi_2(v)| \leq |v| .
 \end{align*}

 \begin{proof}[Proof of Step 1] 
 With $v=a\xi^\perp$ for some $a\in\R$ by the assumption $\xi\perp R$ and $t=\frac{|\xi-\eta|}{|v|}$ we find
 \begin{align*}
  \left|\la (1+t)v ,\eta^\perp \ra\right| &=  |(1+t)a| \cdot \left|\la \xi^\perp ,\eta^\perp \ra\right| \\
  &= \frac{|v| + |\xi-\eta|}{|\xi|} \cdot \left|\la \xi ,\eta \ra\right|  \\
  &= \frac{|v| + |\xi-\eta|}{|\xi|} \cdot \left|\la \xi ,\eta - \xi + \xi \ra\right|  \\
  &= (|v| + |\xi-\eta|) \cdot |\xi| .
 \end{align*}
 In the last equality we used $\eta -\xi \perp \xi$ due to $\xi\perp R$. Similarly, using $\xi=\eta+tv$ due to \eqref{eq orientation} we obtain
 \begin{align*}
  |\eta|^2 + \la (1+t)v, \eta \ra & = \la \xi+v, \eta \ra \\
  & = \la \xi+v, \xi-tv \ra \\
  & = \la \xi+v, \xi \ra -t \la \xi+v, v \ra \\
  & = | \xi |^2 - |\xi-\eta||v| .
 \end{align*}
 Recalling the assumption $k(z)=|z|^2$, we find that \eqref{cond gaussian} is equivalent to
 \begin{align}\label{cond gaussian 1}
    \frac{\frac{|v|}{|\xi|} + \frac{|\xi-\eta|}{|\xi|}}{1 -\frac{|v|}{|\xi|} \cdot \frac{|\xi-\eta|}{|\xi|}} \leq \frac{1 }{ |z|^2 } .
 \end{align}
 Observe that in general, for $0\leq x,y\leq 1/2$ one can verify that
 \begin{align*}
  \frac{x+y}{1-xy} \leq x+y + 2(x^2+y^2) .
 \end{align*}
 Hence, using \eqref{formula xi-eta} and \eqref{cond v} we find for the left hand side of \eqref{cond gaussian 1}
 \begin{align*}
  A:= \frac{\frac{|v|}{|\xi|} + \frac{|\xi-\eta|}{|\xi|}}{1 -\frac{|v|}{|\xi|} \cdot \frac{|\xi-\eta|}{|\xi|}} & \leq \frac{|v|}{|\xi|} + \frac{|\xi-\eta|}{|\xi|} + 2\left(\left(\frac{|v|}{|\xi|}\right)^2 + \left(\frac{|\xi-\eta|}{|\xi|}\right)^2\right) \\
  & \leq \frac{1}{q^4} + \frac{1}{q(q+1)} \frac{1}{1+ r\om} + 2\left( \frac{1}{q^8} + \frac{1}{q^4} \frac{1}{(1+ r\om)^2}\right) . 
 \end{align*}
 On the other hand, by $z=(q,p)$, $r=p/q$ and \eqref{estimate diophantine}, we have for the right hand side of \eqref{cond gaussian 1}
 \begin{align*}
  \nonumber \frac{1}{|z|^2} & = \frac{1}{q^2(1+r^2)} \\ 
  & = \frac{1}{q^2} \cdot \frac{1}{1+r\om+r(r-\om)} \\
  \nonumber & \geq \frac{1}{q^2} \cdot \frac{1}{1+r\om+|r| \frac{1}{q(q+1)} } .
 \end{align*}
 Put together, we obtain
 \begin{align*}
  & q^2(q+1) \cdot (\frac{1}{|z|^2} - A) \\
  & \geq \frac{q+1}{1+r\om+|r| \frac{1}{q(q+1)} } - \left[ \frac{q+1}{q^2} + \frac{q}{1+ r\om} + 2\frac{q+1}{q^2}\left( \frac{1}{q^4} + \frac{1}{(1+ r\om)^2}\right) \right] \\
  & = \frac{q+1}{1+r\om+|r| \frac{1}{q(q+1)}} - \frac{q}{1+ r\om} - \frac{q+1}{q^2} \left[ 1 + 2\left( \frac{1}{q^4} + \frac{1}{(1+ r\om)^2}\right) \right] \\
  & \geq \frac{(q+1)(1+ r\om) - q(1+r\om+|r| \frac{1}{q(q+1)})}{(1+r\om+|r| \frac{1}{q(q+1)})(1+ r\om)} - 5\frac{q+1}{q^2} \\
  & = \frac{1+ r\om -  \frac{|r|}{q+1}}{(1+r\om+|r| \frac{1}{q(q+1)})(1+ r\om)} - 5\frac{q+1}{q^2} .
 \end{align*}
 For $q\to\infty$, the last term converges to $(1+\om^2)^{-1}>0$, proving that \eqref{cond gaussian 1} is fulfilled for sufficiently large $q$.
 \end{proof}
 
 This finishes the proof of Theorem \ref{slope-lemma}. 
\end{proof}

\section{The proof of Proposition \ref{prop lecalvez intro}}\label{section lecalvez}

In this section we sketch the proof of Proposition \ref{prop lecalvez intro}. It stated that for conformally generic Finsler metrics (see Definition \ref{def generic}) there exists an open and dense subset $U\subset S^1$, containing the points with rational and infinite slope, such that the set $\M(\xi)$ is uniformly hyperbolic for all $\xi\in U$. The central arguments were given by P.\ LeCalvez in \cite{lecalvez}.

First, recall that by Proposition \ref{generic-paper result}, the shortest closed geodesic in each homotopy class is unique, if $F$ is conformally generic. The main property of our Finsler metric that we need for the proof will be that it is a ``Kupka-Smale metric''. This is formulated in the following theorem, which is a variant of Theorem D in \cite{contreras-conj} due to G.\ Contreras and R.\ Iturriaga.

\begin{thm}\label{thm kupka-smale}
 The following property is conformally generic for Finsler metrics on $\T$:
 \begin{itemize}
  \item in all non-trivial free homotopy classes of $\T$, the unique shortest closed geodesic is hyperbolic and its stable and unstable manifolds intersect transversely.
 \end{itemize}
\end{thm}

We will not prove Theorem \ref{thm kupka-smale} here. The interested reader is referred to the proof of Theorem D in \cite{contreras-conj}, noting that the perturbation $L+\phi$ is replaced by $\phi\cdot L$, where the Lagrangian $L$ is given by $L=\frac{1}{2}F^2$.

Next, we can use the arguments in \cite{lecalvez} to show the following.

\begin{thm}\label{thm lecalvez rational}
 Assuming the assertion in Theorem \ref{thm kupka-smale}, the set $\M(\xi)$ is hyperbolic for all $\xi\in S^1$ with rational or infinite slope.
\end{thm}

\begin{proof}[Sketch of the proof]
 We only sketch the proof, referring to \cite{lecalvez} for the details. Roughly, the argument is a follows. By assumption, the set of periodic minimal geodesics $\M^{per}(\xi)$ is a single hyperbolic orbit with (un)stable manifolds intersecting transversely. Then one notes that $A := \M(\xi)-\M^{per}(\xi)$ is a subset of the (un)stable manifolds of the hyperbolic periodic orbit (see Theorem \ref{morse periodic}). An application of the $\lam$-Lemma shows that the closure of each orbit in $A$ is uniformly hyperbolic. Then one uses the transversality of the intersection of (un)stable manifolds to infer that $A$ consists of only finitely many orbits. Thus, $\M(\xi)$ is hyperbolic.
\end{proof}

Using the well-known stability of hyperbolic sets, we can now prove Proposition \ref{prop lecalvez intro}.

\begin{proof}[Proof of Proposition \ref{prop lecalvez intro}]
 We show that for each $\xi\in S^1$ with rational or infinite slope, there exists a small neighborhood $U_\xi\subset S^1$ of $\xi$, such that $\M(\eta)$ is hyperbolic for all $\eta\in U_\xi$. Assuming the contrary, fix $\xi$ with rational slope and let $\xi_n\to\xi$, such that $\M(\xi_n)$ is not hyperbolic. By the upper semi-continuity of $\xi\mapsto\M(\xi)$ (Lemma \ref{M upper semi-cont}), we find by Proposition 6.4.6 on p.\ 265 in \cite{KH}, that $\M(\xi_n)$ has to be hyperbolic for large $n$, which is a contradiction.
\end{proof}

\section{Examples}\label{section examples}


\subsection{The flat torus}

The simplest example is the flat torus. More generally, consider any Finsler metric $F:T\T\to\R$, which does not depend on the base variable $x\in \T$ (in standard coordinates $T\T=\T\times \R^2$). Here we find that
\[ \sig_F = F , \]
which is everywhere strongly convex by the definition of a Finsler metric. By Main Theorem \ref{main thm intro irrat} \eqref{main irrat i}, this corresponds to KAM-tori in $S\T$. Indeed, these are simply given by
\[ \M(\xi) = \{ (x,v) : x\in\T, v = \xi/F(\xi) \}\subset S\T. \]
A special case is the euclidean norm $F(v)=|v|$, where
\[ \{\sig_F=1\}= S^1 . \]

\subsection{The rotational torus}

The next well-known example is a rotational torus in $\R^3$, obtained by rotating a circle in the $x_1$-$x_3$-plane about the $x_3$-axis. Here, one has the ``inner'' closed geodesic as a hyperbolic closed geodesic, such that $\M^{per}(\pm e_1)$ become hyperbolic sets and Theorem \ref{intro rational hyp} applies. In the following, we will find a formula for drawing the unit circle $\{\sig_F=1\}\subset\R^2$ with a computer.

Let $c=(c_1,c_2,c_3):\R\to\R^3$ be a curve with $c_1>0,c_2=0$ and $|\dot c|=1$. We can parametrize a surface of revolution in $\R^3$ via
\[ \vf:\R^2\to\R^3, \qquad \vf(s,t) = \left(\begin{smallmatrix} \cos(s) & -\sin(s) & 0 \\ \sin(s) & \cos(s) & 0 \\ 0 & 0 & 1 \end{smallmatrix}\right) \cdot \left(\begin{smallmatrix} c_1(t) \\ 0 \\ c_3(t) \end{smallmatrix}\right) . \]
If $h:\R\to\R$ is a solution of $h'=c_1\circ h$, then $h$ is strictly increasing by $c_1>0$. We find with
\[ \tilde\vf(x_1,x_2) := \vf(x_1,h(x_2)) , \]
that for the pullback of the euclidean metric $\skp_{\R^3}$
\[ (\tilde \vf^*\skp_{\R^3})_x = f(x_2) \cdot \skp , \qquad f := (c_1\circ h)^2. \]
Moreover, if $c$ is a periodic curve, then the function $f$ will be periodic.

A Riemannian metric $g$ on $\R^2$ of the form
\[ g_x(v,w) = f(x_2) \cdot \la v,w \ra , \qquad x=(x_1,x_2) \]
is called a {\em rotational metric}. One can draw the level set $\{\sig_g=1\}$ using a computer and the formula in the following theorem, see Figure \ref{fig_al-levelset}.

\begin{thm}\label{thm example rotational}
 If $g_x=f(x_2)\cdot \skp$ is a rotational metric on $\T=\R^2/\Z^2$, then the unit circle $\{\sig_g=1\}$ of the stable norm is given by the union of the two curves
 \[ t \qquad \mapsto \qquad \pm \frac{1}{\int_0^1 \frac{f(x_2)}{\sqrt{f(x_2) - t^2}} ~dx_2  } \cdot  \left( \int_0^1 \frac{t}{\sqrt{f(x_2) - t^2}} ~dx_2 , 1  \right) \]
 with $|t|\leq\sqrt{\min f}$.
\end{thm}

\begin{figure}\centering 
 \includegraphics[scale=0.7]{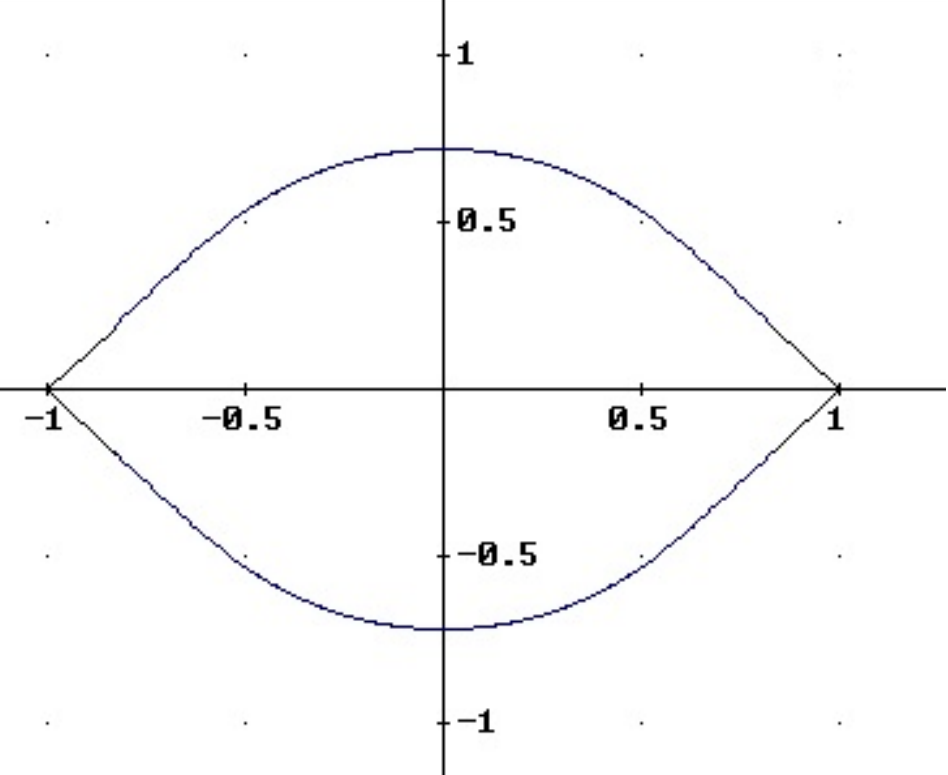}
 \caption{The level set $\{\sig_g=1\}$ for the rotational metric $g=f\cdot \skp$ with $f(x)=\cos(2\pi x)+2$. Note that $\{\sig_g=1\}$ has tangencies to infinite order to its one-sided tangent spaces at the points $\pm e_1$ (the parts of $\{\sig_g=1\}$ look like a straight line). This corresponds via Theorem \ref{intro rational hyp} to the hyperbolicity of $\M^{per}(\pm e_1)$. On the other hand, $\sig_g$ is strongly convex in $\R^2-\R e_1$, corresponding to $C^\infty$-KAM-tori via Main Theorem \ref{main thm intro irrat} \eqref{main irrat i}. See also the proof of Theorem \ref{thm example rotational}.\label{fig_al-levelset}}
\end{figure}

\begin{proof}
First, we move to the Hamiltonian setting via the Legendre transform
\[ \LL:T\T\to T^*\T, \qquad \LL(x,v) = g_x(v,.) = f(x_2) \cdot \la v,. \ra, \]
where the geodesic flow of $g$ is described in standard coordinates of $T^*\T=\T\times \R^2$ by the Hamiltonian
\[ H(x,p) = \frac{|p|^2}{2\cdot f(x_2)} , \]
which is dual to the Lagrangian $L(x,v)=\frac{1}{2}g_x(v,v)$. $H$ admits the coordinate function $p_1$ as a first integral, such that for $(a,b)\in\R^2$ the sets
\[ \Sigma_{a,b} := \{ H=a,~ p_1=b \} \subset T^*\T\]
are invariant under the Hamiltonian flow of $H$. We find
\[ H(x,p)=a ~\&~ p_1=b \quad \iff \quad p=p(x) = \left( b , \pm\sqrt{2a f(x_2) - b^2} \right) . \]
For the case $2a \min f > b^2$, the above formula defines two smooth, invariant, Lagrangian graphs
\[ \Sigma_{a,b}^\pm :=\{(x,p(x)):x\in \T\} \]
(which are in fact $C^\infty$-KAM-tori in the sense of Definition \ref{def KAM intro}). We can write $p(x)$ seen as a closed 1-form on $\T$ in the form $p=\eta+du$, where $\eta\in (\R^2)^* \cong H^1(\T,\R)$ is a constant 1-form and $u:\T\to\R$ is some function (see e.g.\ Lemma 3.4 in \cite{diss}); $\eta=[p]$ is called the Liouville class of the Lagrangian graph. We thus find a formula for the Liouville class of the graph $\Sigma_{a,b}^\pm$ given by
\[ [\Sigma_{a,b}^\pm] =\eta = \int_{\T} p(x) dx = \left( b , \int_0^1 \pm \sqrt{2a f(x_2) - b^2} ~dx_2 \right). \]

We shall use a bit of language from Mather theory; see e.g.\ \cite{sorrentino} for an introduction. Namely, the convex dual of Mather's $\beta$-function $\beta=\frac{1}{2}\sig_g^2$ is Mather's $\al$-function. Note that $\al$ is a $C^1$-function by $\beta$ being strictly convex in the case of $\T$, as we already noted in the introduction. Moreover, if $\eta$ is the Liouville class of an invariant Lagrangian graph, then $\al(\eta)$ equals the energy of the corresponding graph. Hence,
\[ \al([\Sigma_{a,b}^\pm]) = H|_{\Sigma_{a,b}^\pm} = a . \]
This shows
\[ \left\{\left( b , \int_0^1 \pm \sqrt{2a f(x_2) - b^2} ~dx_2 \right) :  b^2 < 2a \min f \right\} \subset \{\al = a \} . \]
Next, observe that by Fenchel duality and $\al(t\eta)=t^2\al(\eta)$ we have for the euclidean gradient $\nabla\al$ of $\al$, that
\[ \beta(\nabla\al(\eta)) = \la \nabla\al(\eta), \eta \ra - \al(\eta) = \al(\eta) \quad \forall \eta . \]
Hence, points in $\{\al=a\}$ together with the gradient $\nabla\al$ correspond to points in 
\[ \{\beta=a\} = \{ \sig_g = \sqrt{2a} \}. \]
Hence, our aim is to compute $\nabla\al(\eta)$ for $\eta\in \{\al=a\}$.

We fix the value $a=1/2$ and consider the function
\[ g(t) := \int_0^1 \sqrt{f(x_2) - t^2} ~dx_2 , \qquad |t| < \sqrt{\min f} . \]
Then the curve
\[ \g(t) := \left( t , g(t) \right) , \qquad |t| < \sqrt{\min f} \]
parametrizes the upper half of the set found in $\{\al=1/2\}$ above. The velocity vector $\dot \g = \left( 1 , g' \right)$ is orthogonal to $\nabla\al\circ \g$, i.e.\ we find some function $\lam(t)$ with
\begin{align*}
 \nabla\al(\g(t)) = \lam(t) \cdot \left( -g'(t) , 1  \right) .
\end{align*}
Using $\la \nabla\al(\eta),\eta\ra = 2\al(\eta)$ due to homogeneity of degree 2, we find
\begin{align*}
 1 & = 2\al(\g(t)) = \la \nabla\al(\g(t)),\g(t)\ra = \lam(t) \cdot (g(t)-t\cdot g'(t)) .
\end{align*}
Hence, for $|t|<\sqrt{\min f}$
\begin{align*}
 \nabla\al(\g(t)) = \frac{1}{g(t)-t\cdot g'(t)} \cdot  \left( -g'(t) , 1  \right) \quad \in \quad \{\beta=1/2\}.
\end{align*}
Next, observe that the second component of $\nabla\al(\g(t))$ vanishes as $|t|\to \sqrt{\min f}$. Indeed, one computes
\[ g(t)-t\cdot g'(t) = \int_0^1 \frac{f(x_2)}{\sqrt{f(x_2) - t^2}} ~dx_2 \geq \min f \cdot \int_0^1 \frac{1}{\sqrt{f(x_2) - t^2}} ~dx_2 . \]
On the other hand,
\begin{align*}
 \int_0^1 \frac{1}{\sqrt{f(x_2) - t^2}} ~dx_2 \to \infty , \quad \text{as } |t|\nearrow \sqrt{\min f} .
\end{align*}
This shows that the two segments $\pm \nabla\al\circ\g(t)$ with $|t|\leq \sqrt{\min f}$ form a closed curve in $\{\beta=1/2\}$. The theorem follows.
\end{proof}

\subsection{The punctured torus}

Here we state a result due to G. McShane and I. Rivin, see \cite{rivin1} and \cite{rivin}. These authors treat hyperbolic metrics $g$ on the punctured torus $\dot T^2$. While this case does not quite fit into our setting of a metric on the closed torus $\T$, the same phenomena appear. Let us state Theorem 2.1 from \cite{rivin}, see also Figure \ref{fig_rivin} taken from \cite{rivin0}.

\begin{thm}\label{thm rivin}
 Let $g$ be a hyperbolic metric on the once punctured torus $\dot T^2$ with finite area. Then the stable norm $\sig_g$ of $g$ on $H_1(\dot T^2,\R)\cong \R^2$ is flat to infinite order at points $\xi\in S^1$ of irrational slope. At points $\xi\in S^1$ of rational or infinite slope, the stable norm is not differentiable and the analogous statement holds on each side of the line $\R_+\xi$.
\end{thm}

Of course, in this case there are no KAM-tori and the geodesic flow is hyperbolic due to negative curvature. In this light, Theorem \ref{thm rivin} confirms Main Theorem \ref{main thm intro irrat} \eqref{main irrat ii} and Theorem \ref{intro rational hyp}.

 \begin{figure}\centering 
 \includegraphics[scale=0.7]{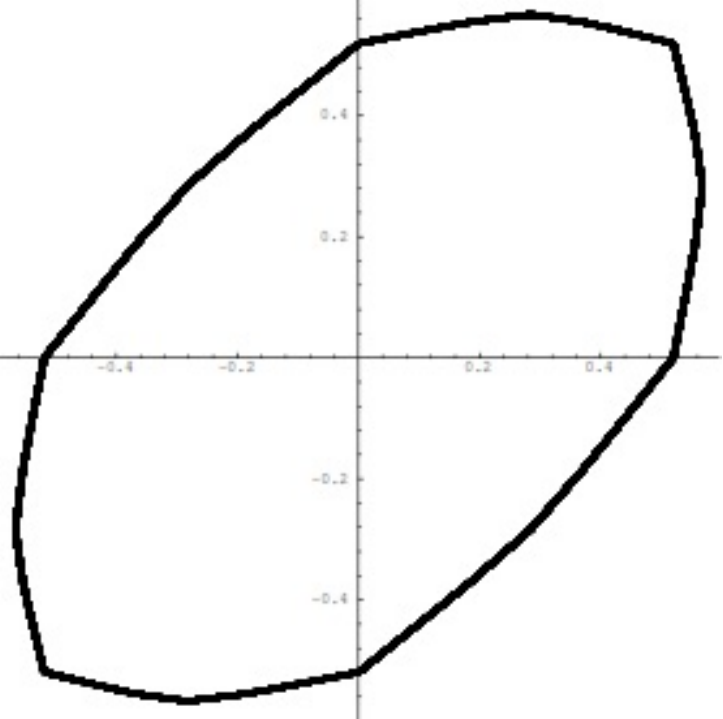}
 \caption{The level set $\{\sig_g=1\}$ in Theorem \ref{thm rivin} for the modular torus, taken from \cite{rivin0}. Even though the level set looks polygonal, it is in fact strictly convex. \label{fig_rivin}}
 \end{figure}



\begin{thebibliography}{GKOS14}
 \bibitem[Arn11]{arnaud} M.-C. Arnaud -- {\em The link between the shape of irrational Aubry-Mather sets and their Lyapunov exponents}. Annals of Mathematics 174 (2011), 1571-1601.
 
 \bibitem[Arn13]{arnaud-example} M.-C. Arnaud -- {\em Boundaries of instability zones for symplectic twist maps}. Journal of the Institute of Mathematics of Jussieu 13.1 (2013), 19-41.
 
 \bibitem[AB14]{arnaud non-hyp} M-C. Arnaud, P. Berger. -- {\em The non-hyperbolicity of irrational invariant curves for twist maps and all that follows.} arXiv:1411.7072 (2014).
 
 \bibitem[Ban88]{bangert} V. Bangert -- \emph{Mather sets for twist maps and geodesics on tori}. Dynamics Reported 1 (1988), 1-56.
 
 \bibitem[Ban89]{bangert-minimal-geod} V. Bangert -- {\em Minimal geodesics}. Ergodic Theory and Dynamical Systems 10 (1989), 263-286.
 
 \bibitem[Ban94]{bangert1} V. Bangert -- \emph{Geodesic rays, Busemann functions and monotone twist maps}. Calculus of Variations and Partial Differential Equations 2.1 (1994), 49-63.
 
 \bibitem[BCS00]{BCS} D. D.-W. Bao, S. S. Chern, Z. Shen -- \emph{An introduction to Riemann-Finsler geometry}. Graduate Texts in Mathematics 200, Springer Verlag (2000).
 
 \bibitem[Bre93]{bredon} G.\ E.\ Bredon -- {\em Topology and Geometry}. Graduate Texts in Mathematics 139, Springer Verlag (1993).
 
 \bibitem[BQ07]{bressaus-quas} X. Bressaud, A. Quas -- {\em Rate of approximation of minimizing measures}. Nonlinearity 20.4 (2007), 845-853.
 
 
 \bibitem[CR06]{carneiro} M. J. D. Carneiro, R. O. Ruggiero -- {\em On Birkhoff Theorems for Lagrangian invariant tori with closed orbits}. Manuscripta Mathematica 119.4 (2006), 411-432.
 
 \bibitem[CI99]{contreras-conj} G. Contreras, R. Iturriaga -- {\em Convex Hamiltonians without conjugate points}. Ergodic Theory and Dynamical Systems 19.4 (1999), 901-952.
 
 \bibitem[CIPP98]{cipp} G. Contreras, R. Iturriaga, G. P. Paternain, M. Paternain -- \emph{Lagrangian graphs, minimizing measures and \mane's critical values}. Geometric and Functional Analysis 8 (1998), 788-809.
 
 
 
 
 \bibitem[Hed32]{hedlund} G. A. Hedlund -- \emph{Geodesics on a two-dimensional Riemannian manifold with periodic coefficients}. The Annals of Mathematics 33.4 (1932), 719-739.
 
 \bibitem[Her83]{herman} M. R. Herman -- \emph{Sur les courbes invariantes par les diff\'eomorphismes de l'anneau, Volume 1}. Asterisque 103-104 (1983).
 
 
 \bibitem[Hop48]{hopf} E. Hopf -- {\em Closed surfaces without conjugate points}. Proceedings of the National Academy of Sciences of the United States of America 34.2 (1948), 47-51.
 
 \bibitem[KH95]{KH} A. Katok, B. Hasselblatt -- \emph{Introduction to the modern theory of dynamical systems}. Encyclopedia of Mathematics and its Applications 54, Cambridge University Press (1995).
 
 
 \bibitem[LeC88]{lecalvez} P. LeCalvez -- {\em Les ensembles d'Aubry-Mather d'un diff\"eomorphisme conservatif de l'anneau d'eviant la verticale sont en g\'en\'eral hyperboliques}. Comptes Rendus de l'Acad\'emie des Sciences - Series I - Mathematics, 306.1 (1988), 51-54.
 
 \bibitem[Mac92]{mckay} R. S. MacKay -- {\em Greene's residue criterion}. Nonlinearity 5 (1996), 161-187.
 
 \bibitem[Ma\~n96]{mane} R. \Mane -- \emph{Generic properties and problems of minimizing measures of Lagrangian systems}. Nonlinearity 9 (1996), 273-310.
 
 \bibitem[Mas96]{massart-thesis} D. Massart -- {\em Normes stables des surfaces}. Th\'ese doctorat, Ecole Normale Sup\'erieure de Lyon (1996).
 

 \bibitem[Mas03]{massart2} D. Massart -- {\em On Aubry sets and Mather's action functional}. Israel Journal of Mathematics 134 (2003), 157-171.
 
 \bibitem[Mas15]{massart-erratum} D. Massart -- {\em Erratum to ``On Aubry sets and Mather's action functional'', Israel Journal of Mathematics 134 (2003), 157-171}. Israel Journal of Mathematics 207 (2015), 1001.

 \bibitem[MS11]{massart-sorrentino} D. Massart, A. Sorrentino -- \emph{Differentiability of Mather's average action and integrability on closed surfaces}. Nonlinearity 24.6 (2011), 1777-1793.
 
 \bibitem[Mat88]{mather2} J. N. Mather -- {\em Destruction of Invariant Circles}. Ergodic Theory and Dynamical Systems 8 (1988), 199-214.

 \bibitem[Mat90]{mather1} J. N. Mather -- {\em Differentiability of the minimal average action as a function of the rotation number}. Bol. Soc. Bras. Mat. 21.1 (1990), 59-70.
 
 \bibitem[MF94]{mather-forni} J. N. Mather, G. Forni -- \emph{Action minimizing orbits in Hamiltonian systems}. Transition to chaos in classical and quantum mechanics, Springer Verlag (1994), 92-186.
 

 \bibitem[MR95a]{rivin1} G. McShane, I. Rivin -- {\em A norm on homology of surfaces and counting simple geodesics}. International Mathematics Research Notices 2 (1995), 61-69.

 \bibitem[MR95b]{rivin} G. McShane, I. Rivin -- {\em Simple curves on hyperbolic tori}. Comptes Rendus de l'Acad\'emie des Sciences - Series I - Mathematics, 320.12 (1995), 1523-1528.

 \bibitem[MR00]{rivin0} G. McShane, I. Rivin -- {\em Simple curves on hyperbolic tori}. arXiv:math/0005220 [math.GT] (2000).

 \bibitem[Mor24]{morse} H. M. Morse -- \emph{A fundamental class of geodesics on any closed surface of genus greater than one}. Transactions of the American Mathematical Society 26.1 (1924), 25-60.
 
 \bibitem[Mos62]{moser} J. Moser -- \emph{On invariant curves of area-preserving mappings of an annulus}. Nachrichten der Akademie der Wissenschaften in G\"ottingen, Mathematisch-Physikalische Klasse II 1962.1 (1962), 1-20..
 
 
 
 
 \bibitem[Sch13]{diss} J.\ P.\ Schr\"oder -- \emph{Tonelli Lagrangians on the 2-torus: global minimi\-zers, invariant tori and topological entropy}. Ph.D. thesis, Ruhr-Universit\"at Bochum (2013). Available online at {\scriptsize \url{http://www.ruhr-uni-bochum.de/ffm/Lehrstuehle/Lehrstuhl-X/jan.html}}.
 
 \bibitem[Sch15a]{paper1} J. P. Schr\"oder -- {\em Global minimizers for Tonelli Lagrangians on the 2-torus}. Journal of Topology and Analysis 7.2 (2015), 261-291.
 
 \bibitem[Sch15b]{min_rays} J. P. Schr\"oder -- \emph{Minimal rays on closed surfaces}. Preprint (2015), to appear in Israel Journal of Mathematics.
 
 \bibitem[Sch15c]{generic-paper} J. P. Schro\"oder -- \emph{Generic uniqueness of geometric minimizers}. Preprint (2015).
 
 \bibitem[Sib00]{siburg-paper} K. F. Siburg -- {\em Symplectic invariants of elliptic fixed points}. Commentarii Mathematici Helvetici 75 (2000), 681-700.

 \bibitem[Sib04]{siburg} K. F. Siburg -- {\em The principle of least action in geometry and dynamics}. Lecture Notes in Mathematics 1844, Springer Verlag (2004).

 \bibitem[Sor10]{sorrentino} A. Sorrentino -- \emph{Lecture notes on Mather's theory for Lagrangian systems}. arXiv: 1011.0590 [math.DS] (2010).
 
 \bibitem[Zau62]{zaustinsky} E. M. Zaustinsky -- \emph{Extremals on compact {\it E}-surfaces}. Transactions of the American Mathematical Society 102.3 (1962), 433-445.
\end{thebibliography}
\end{document}